\documentclass{amsart}

\usepackage{epsfig}
\usepackage{amsmath}
\usepackage{amssymb}
\usepackage{wasysym}
\usepackage{amscd}
\usepackage{graphicx}
\usepackage{xcolor}
\usepackage{float}
\usepackage{verbatim}
\usepackage{bbm}
\usepackage{cancel}
\usepackage{enumerate}

\newtheorem{thm}{Theorem}[section]
\newtheorem{lem}[thm]{Lemma}
\newtheorem{prop}[thm]{Proposition}
\newtheorem{fact}[thm]{Fact}

\newtheorem{cor}[thm]{Corollary}
\newtheorem{defn}[thm]{Definition}

\theoremstyle{remark}
\newtheorem{rem}[thm]{Remark}

\newtheorem{exam}[thm]{Example}

\def \N {\mathbb N}
\def \T {\mathcal T}
\def \TT {\mathsf T}
\def \TTT {\boldsymbol{\TT}}

\def \B {\mathcal B}

\def \Z {\mathbb Z}

\def \F {\mathcal F}

\def \P {\mathcal P}
\def \S {\mathcal S}

\def \htop {h_{\mathsf{top}}}

\def \eps {\varepsilon}

\def \sq {sequence}

\def \tl {topological}
\def \im {invariant measure}
\def \inv {invariant}

\def \htop{h_{\mathsf{top}}}

\numberwithin{equation}{section}

\begin{document}

\title{Universality of $G$-subshifts with specification}

\author{Tomasz Downarowicz, Benjamin Weiss, Mateusz Wi\c ecek, Guohua Zhang}

\address{\vskip 2pt \hskip -12pt Tomasz Downarowicz}

\address{\hskip -12pt Faculty of Pure and Applied Mathematics, Wroc\l aw University of Technology, Wroc\l aw, Poland}

\email{downar@pwr.edu.pl}
\medskip
\address{\vskip 2pt \hskip -12pt Benjamin Weiss}

\address{\hskip -12pt Einstein Institute of Mathematics,
The Hebrew University of Jerusalem}

\email{weiss@math.huji.ac.il}
\address{\vskip 2pt \hskip -12pt Mateusz Wi\c ecek}

\address{\hskip -12pt Faculty of Pure and Applied Mathematics, Wroc\l aw University of Technology, Wroc\l aw, Poland}

\email{mateusz.wiecek@pwr.edu.pl}
\medskip
\address{\vskip 2pt \hskip -12pt Guohua Zhang}

\address{\hskip -12pt School of Mathematical Sciences and Shanghai Center for Mathematical Sciences, Fudan University, Shanghai 200433, China}

\email{chiaths.zhang@gmail.com}

\begin{abstract}Let $G$ be an infinite countable amenable group and let $(X,G)$ be a $G$-subshift with specification, containing a free element. We prove that $(X,G)$ is universal, i.e., has positive \tl\ entropy and for any free ergodic $G$-action on a standard probability space, $(Y,\nu,G)$, with $h(\nu)<\htop(X)$, there exists a shift-\im\ $\mu$ on $X$ such that the systems $(Y,\nu,G)$ and $(X,\mu,G)$ are isomorphic. In particular, any $K$-shift (consisting of the indicator functions of all maximal $K$-separated sets) containing a free element is universal.
\end{abstract}

    \thanks{The research of Guohua Zhang is supported by the National Key Research and Development Program of China Grant No. 2021YFA1003204, and the New Cornerstone Science Foundation through the New Cornerstone Investigator Program. The research of Mateusz Wi\c ecek is supported by the NCN Grant No. 2021/41/N/ST1/02816. The research of Tomasz Downarowicz is supported by the NCN Grant No. 2022/47/B/ST1/02866.}

\maketitle

\section{Motivation, preliminaries and statement of the main result}
This paper is concerned with properties of specific symbolic $G$-systems, where $G$ is an infinite countable amenable group. In order to describe the goal of this paper, we need to establish some elementary terminology. Throughout the paper, $G$ always denotes an infinite countable group with unity denoted by $e$. Most of the time, we assume that $G$ is amenable, although this does not apply to the following two subsections.
\subsection{Symbolic systems}
\begin{defn}\label{shift}
Let $\Lambda$ be a compact countable set containing at least two elements, called \emph{alphabet}. The $G$-shift is the dynamical system $(\Lambda^G,G)$ where $G$ acts on
$\Lambda^G$ as follows: if $x=(x(g))_{g\in G}\in\Lambda^G$ and $h\in G$ then
$$
h(x)=(x(gh))_{g\in G}.
$$
Any subsystem $(X,G)$ of the $G$-shift (i.e., $X$ is a nonempty, closed in the product topology and shift-\inv\ subset of $\Lambda^G$) is called a \emph{$G$-subshift} or a \emph{symbolic $G$-system}. When the group $G$ is fixed, we will often skip the prefix ``$G$-'' and say ``shift'', ``subshift'' or ``symbolic system''.
\end{defn}

Unless otherwise specified, the alphabets in the considered subshifts will be assumed to be finite.
\smallskip

Elements of $\Lambda^G$ will be called (\emph{symbolic}) \emph{elements}. If $x=(x(g))_{g\in G}$ is a symbolic element and $A\subset G$ then the restriction of $x$ to $A$ will be denoted by $x|_A$. By a \emph{pattern over $A$} we will mean any function $\alpha:A\to\Lambda$. If $A$ is finite and contains $e$ then any pattern over $A$ will be called a \emph{block} and $e$ will be called the \emph{center} of the block. In what follows, we will write $\alpha\approx\alpha'$ for two patterns $\alpha,\alpha'$ with domains $A$ and $A'$, respectively, if they are ``equal up to a shift'', that is, there exists some $g\in G$ such that $A'=Ag$ and for every $a\in A$ we have $\alpha'(ag)=\alpha(a)$. 
Let $x=(x(g))_{g\in G}\in\Lambda^G$ be a symbolic element, let $A\subset G$ and let $\alpha:A\to\Lambda$ be a pattern over $A$. We will say that $\alpha$ \emph{occurs} in $x$ if there exists $g\in G$ such that $x|_{Ag}\approx\alpha$. If $\alpha$ is a block and we need to be more precise, we will say that $\alpha$ occurs in $x$ \emph{centered at} $g$ and $g$ will be referred to as \emph{center} of this occurrence. If $(X,G)$ is a subshift and $\alpha$ is a pattern over some set $A\subset G$ then we will say that this pattern is $X$-admissible if $\alpha$ occurs in some $x\in X$.

\subsection{$K$-shifts}
Let $K$ be a finite subset of $G$, containing the unity $e$ and at least one more element. 
\begin{defn}\label{ksep}\phantom{a}
\begin{itemize}
	\item A set $V\subset G$ is \emph{$K$-separated} if the sets $Kg$ are disjoint as $g$ ranges~over~$V$. 
	\item A $K$-separated set $V$ is \emph{maximal} if $V\cup\{g\}$ is not $K$-separated for any $g\notin V$.
\end{itemize}
\end{defn}

Note the obvious fact that every $K$-separated set is contained in some maximal $K$-separated set. It is not hard to see that a $K$-separated set $V$ is maximal if and only if $K^{-1}KV=G$. In other words, a set $V\subset G$ is maximal $K$-separated if and only if it has the following two properties:
\begin{enumerate}[(I)]
	\item every set of the form $K^{-1}Kg$ ($g\in G$) contains at least one element of $V$,
	\item if $g\in V$ then $K^{-1}Kg$ contains no other elements of $V$.
\end{enumerate}

\begin{defn}\label{kshift}
The \emph{$K$-shift} is defined as the collection $\Omega_K\subset\{0,1\}^G$ 
of the indicator functions of all maximal $K$-separated subsets of $G$.
Clearly, $\Omega_K$ is closed and shift-\inv, i.e., it is a subshift on two symbols. 
\end{defn}

\subsection{Amenability, freeness and entropy}
\begin{defn}\label{amen}
Given a finite set $K\subset G$ and $\eps>0$, we say that a finite set $F\subset G$ is \emph{$(K,\eps)$-\inv} if 
$$
\frac{|KF\triangle F|}{|F|}<\eps
$$
($\cdot\triangle\cdot$ stands for the symmetric difference of sets and $|\cdot|$ denotes cardinality).
A \sq\ $\F=(F_n)_{n\in\N}$ of finite subsets of $G$ is called a \emph{F\o lner \sq} if for every finite set $K$ and every $\eps>0$ the sets $F_n$ are eventually $(K,\eps)$-\inv.
A group $G$ is \emph{amenable} if it admits a F\o lner \sq.
\end{defn}
Note that if $\F=(F_n)_{n\in\N}$ is a F\o lner \sq\ in $G$ then $\lim_{n\to\infty}|F_n|=\infty$. 

We take this opportunity to introduce some useful terminology concerning relations between finite subsets of countable groups.

\begin{defn}\phantom{.}
\begin{itemize}
	\item Let $K$ and $F$ be finite subsets of $G$. The \emph{$K$-core} of $F$ is the set
$$
F_K=\{f\in F: Kf\subset F\}.
$$
	\item Fix an $\eps>0$. A set $F'\subset F$ is called a \emph{$(1-\eps)$-subset} of $F$ if $\frac{|F'|}{|F|}\ge1-\eps$. A set $F''\supset F$ is called an \emph{$\eps$-enlargement} of $F$ if $\frac{|F''|}{|F|}\le1+\eps$.
\end{itemize}
\end{defn}

The following easy facts hold for finite subsets of $G$.

\begin{prop}\label{ksets}Assume that $F$ is $(K,\eps)$-\inv, where $K\ni e$. Then:
\begin{enumerate}
	\item For any $g\in K$, $F$ is $(\{g\},2\eps)$-\inv.
	\item $F_K$ is a $(1-|K|\eps)$-subset of $F$.
	\item If $F'$ is a $(1-\eps)$-subset of $F$ then $F'_K$ is a $(1-2|K|\eps)$-subset of $F$.
\end{enumerate}
\end{prop}

We will also need the following auxiliary fact:

\begin{prop}\label{aux}
Fix a $\delta>0$ and let $K,K'$ be finite subsets of $G$ such that $e\in K\subset K'$ and $K'$ is symmetric. If a finite set $F\subset G$ is $(K',\delta)$-\inv\ then, for any $K'$-separated set $V$, we have
$$
\frac{|KV\cap F|}{|F|}\le|K'|^2\delta+\frac{|K|}{|K'|}.
$$
\end{prop}
\begin{proof}
We have $KV\cap F\subset K(V\cap F_{K'})\cup K(F\setminus F_{K'})$. Indeed, if $f=kv\in KV\cap F$ with $k\in K$ and $v\in V\cap F_{K'}$ then $f\in K(V\cap F_{K'})$. Otherwise consider two cases 
\begin{enumerate}[(a)]
    \item If $v\in V\cap(F\setminus F_{K'})$ then $f\in K(F\setminus F_{K'})$. 
    \item Now suppose $v\notin F$. If 
    $kv$ belonged to $F_{K'}$ then $v=k^{-1}kv$ would belong to $K'F_{K'}\subset F$, a contradiction. So, in this case, $f=kv\in F\setminus F_{K'}\subset K(F\setminus F_{K'})$.
\end{enumerate}
By Proposition~\ref{ksets}~(2), we have $\frac{|K(F\setminus F_{K'})|}{|F|}\le\frac{|K|\cdot|F\setminus F_{K'}|}{|F|}\le|K|\cdot|K'|\delta\le |K'|^2\delta$.
Since the elements of $V\cap F_{K'}$ are $K'$-separated and $K'(V\cap F_{K'})\subset F$, the cardinality of $V\cap F_{K'}$ is at most $\frac{|F|}{|K'|}$, and thus
$$
\frac{|K(V\cap F_{K'})|}{|F|}\le\frac{|K|}{|K'|},
$$
and we are done.
\end{proof}

\begin{defn}\label{fr}
A point $x$ in a dynamical system $(X,G)$ is \emph{free} if for any $g\in G\setminus\{e\}$ one has $g(x)\neq x$.
A \tl\ $G$ system $(X,G)$ is \emph{free} if all elements $x\in X$ are free.
A measure-preserving $G$-system $(Y,\nu,G)$ (where $(Y,\nu)$ is a standard probability space\footnote{In the notation of standard probability spaces we will skip the symbol denoting the sigma-algebra. Any topological space is, by default, equipped with the Borel sigma-algebra.} and $G$ acts on $(Y,\nu)$ by measure automorphisms) is \emph{free} if $\nu$-almost every $y\in Y$ is free.
\end{defn}

\begin{defn}\label{htop}
Let $X$ be a $G$-subshift, where $G$ is amenable, with a F\o lner \sq\ $\F=(F_n)_{n\in\N}$. The \emph{\tl\ entropy} of $X$ is defined as follows: 
$$
\htop(X)=\lim_{n\to\infty}\frac1{|F_n|}\log|\{\alpha\in\Lambda^{F_n}:\ \alpha\text{ occurs in some }x\in X\}|
$$
(it is known that the limit exists, equals the infimum, and does not depend on the F\o lner \sq -- see e.g.~\cite[Theorem~6.1]{LW}).
\end{defn}

\begin{defn}\label{entropy}
Let $(Y,\nu,G)$ be a measure-preserving $G$-system, where $G$ is amenable, with a F\o lner \sq\ $\F=(F_n)_{n\in\N}$. The \emph{entropy} of $\nu$, denoted by~$h(\nu)$, is defined in three steps. 
\begin{enumerate}
	\item Given a finite measurable partition $\P$ of $Y$, we let
	$$
	H_\nu(\P)=-\sum_{P\in\P}\nu(P)\log(\nu(P)).
	$$
	\item The \emph{entropy of the process} generated by $\P$ is defined as
	$$
	h(\nu,\P)=\lim_{n\to\infty}\frac1{|F_n|}H_\nu\left(\bigvee_{i=0}^{n-1}g^{-1}(\P)\right).
	$$
	\item Finally, we let
	$$
	h(\nu)=\sup_{\P}h(\nu,\P),
	$$
	where the supremum is taken over all finite measurable partitions of $Y$.
\end{enumerate}
It is known that the limit in (2) exists, equals the infimum, and does not depend on the F\o lner \sq.
\end{defn}
If $\nu$ is a shift-\im\ on a symbolic system, then $h(\nu)=h(\nu,\P_\Lambda)$, where $\P_\Lambda$ is the \emph{one-symbol partition} of $\Lambda^G$ by the sets $[a]=\{x\in X: x(e)=a\}$, $a\in\Lambda$.

\subsection{Initial motivation} 
The original plan of the work that eventually led us to writing this article was to prove that $K$-shifts are \emph{universal} for free ergodic systems in the following sense.

\begin{defn}Let $G$ be amenable and let $X$ be a $G$-subshift with positive \tl\ entropy $\htop(X)$. We say that $X$ is \emph{universal} if, for any free ergodic $G$-system $(Y,\nu,G)$ with entropy $h(\nu)<\htop(X)$, there exists an \im\ $\mu$ on $X$ such that the systems $(X,\mu,G)$ and $(Y,\nu,G)$ are isomorphic (we will briefly say that $\mu$ is \emph{isomorphic} to $\nu$).
\end{defn}

\begin{rem}If we define universality without the restriction to \emph{free} ergodic systems, then no $K$-shift would be universal, as no $K$-shift supports an \im\ concentrated at one point. 
\end{rem}

While working on the proof, we have discovered that the attempted theorem is actually false, which becomes evident in the following simple example.

\begin{exam}
Let $G=\Z\times\Z_6$ (any Abelian group is amenable) and let $K=\{(0,0), (0,1), (0,5)\}$. Then for every $K$-separated set $V$ and each $n\in\Z$, the intersection of $V$ with the coset $\{n\}\times\Z_6$ is either empty or contains just one point $(n,i)$, or consists of two points: $(n,i)$ and $(n,i+3)$. If $V$ is maximal then for each $n$ it has to be the third option (see Figure~\ref{exa1}).
\begin{figure}[h]
\includegraphics[width=11cm]{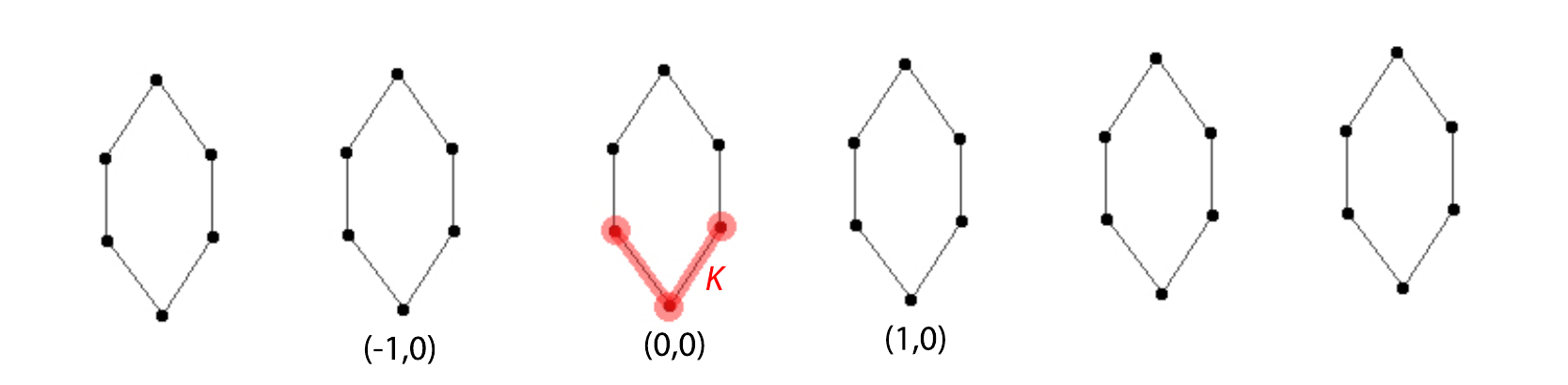}

\includegraphics[width=11cm]{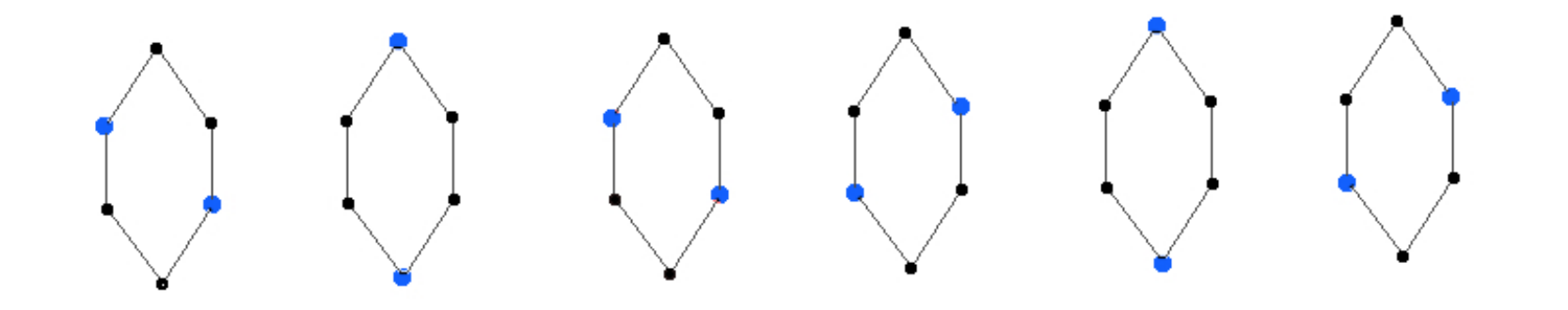}
\caption{Top figure: the group $\Z\times\Z_6$, the set $K$ (in red). Bottom figure: a maximal $K$-separated set (in blue).}\label{exa1}
\end{figure}

But then $V+(0,3)=V$, meaning that $(0,3)$ acts on $\Omega_K$ by identity.
This implies that $\Omega_K$ does not carry free ergodic measures. On the other hand, one easily verifies that $\htop(\Omega_K)=\frac{\log3}6>0$ and it is well known that there exist free measure-preserving $G$-actions with arbitrarily small entropy
(see e.g. \cite[Theorem~6.1]{DHZ} for zero entropy examples). Since the $K$-shift $\Omega_K$ does not support \im s isomorphic to free ergodic measures with entropy less than $\htop(\Omega_K)$, it is not universal.
\end{exam}

\subsection{Specification}
Further investigation revealed that the theorem can be saved by making the assumption that the system $(X,G)$ contains at least one free element. Of course, such a condition is implied by universality, so it is in fact an \emph{equivalent condition}. Then it turned out that under this assumption universality holds in a much larger class of subshifs, namely subshifts with the so-called specification property.

\begin{defn}\label{spec}Let $G$ be any countable group and let $M\subset G$ be a finite set containing $e$. Two sets $A_1,A_2\subset G$ are said to be \emph{$M$-apart} if $MA_1\cap MA_2=\varnothing$. A nontrivial symbolic system $(X,G)$ has the \emph{specification property with margin $M$} if 
whenever $A_1\subset G$ and $A_2\subset G$ are $M$-apart, $\alpha_1, \alpha_2$ are $X$-admissible patterns over $A_1$, $A_2$, respectively, then the pattern $\alpha$ over $A_1\cup A_2$, such that $\alpha|_{A_1}=\alpha_1$ and $\alpha|_{A_2}=\alpha_2$, is $X$-admissible.
\end{defn}

By an obvious induction, the specification property with margin $M$ implies that any (finitely or countably many) $X$-admissible patterns, whose domains are pairwise $M$-apart, jointly form an $X$-admissible pattern. 

The first thing to check is that specification implies positive entropy, which is required in the definition of universality. 

\begin{fact}\label{Omega_entropy}
Let $G$ be amenable. If a nontrivial \footnote{We need at least two $X$-admissible symbols.} subshift $(X,G)$ has the specification property (with some margin $M$) then $\htop(X)>0$.
\end{fact}

\begin{proof}We may assume that $0\in\Lambda$ and $1\in\Lambda$ are $X$-admissible symbols.
Let $V$ be a maximal $M$-separated set. By the specification property, any configuration of $0$'s and $1$'s over $V$ is $X$-admissible. By property (I) of the maximal $M$-separated sets, for any $\eps>0$, if $F_n$ is a sufficiently large member of a F\o lner \sq\ $\F=(F_n)_{n\in\N}$, then the cardinality of $F_n\cap V$ is at least $\frac{|F_n|(1-\eps)}{|M^{-1}M|}$. This easily implies $\htop(X)\ge\frac{\log2}{|M^{-1}M|}$.
\end{proof}

\begin{fact}\label{nomin}
Let $G$ be amenable. A nontrivial subshift $(X,G)$ with specification is never minimal.\footnote{A system $(X,G)$ is minimal if it contains no proper nonempty closed \inv\ subsets.} 
\end{fact}

\begin{proof}
Let $\alpha$ be a block with a finite domain $F\ni e$. Let $K=MF$, where $M$ is the margin of specification of $X$. We define $X_\alpha$ as the collection of all these elements $x\in X$ for which there exists a maximal $K$-separated set $V$ such that  
\begin{equation}\label{av}
x|_{Fv}\approx\alpha, \text{ \ for each }v\in V.
\end{equation}
Because, for each maximal $K$-separated set $V$, the sets $Fv$ with $v\in V$ are $M$-apart, the specification property of $X$ implies that the collection of elements $x\in X$ satisfying~\eqref{av} with a fixed set $V$ is nonempty, and so is $X_\alpha$. 
Since the set $\Omega_K$  consisting of the indicator functions of all maximal $K$-separated sets is closed and shift-\inv, so is $X_\alpha$. Suppose that $(X,G)$ is minimal. Then $X=X_\alpha$ for every finite $X$-admissible block~$\alpha$. Choose a F\o lner \sq\ $\F=(F_n)_{n\ge1}$ so that $F_n\ni e$ for all $n$.
Fix some $n\ge 1$ and let $K=MF_n$. Fix an $x\in X$. Since, for each $X$-admissible block $\alpha$ over $F_n$, $x\in X_\alpha$, $\alpha$ occurs in $x$ centered at the elements of some maximal $K$-separated set. By the property~(I) of maximal $K$-separated sets, $K^{-1}K$ contains a representative of every maximal $K$-separated set, and hence it contains the center of an occurrence in $x$ of every $X$-admissible block $\alpha$ over $F_n$. Since there are $|K^{-1}K|$ centers to choose from, there are at most $|K^{-1}K|\le|M|^2|F_n|^2$ different $X$-admissible blocks $\alpha$ over $F_n$. Thus, according to Definition~\ref{htop}, we have
$$
\htop(X)\le\lim_{n\to\infty}\frac{2\log|M|+2\log|F_n|}{|F_n|}=0,
$$
which is a contradiction with Fact~\ref{Omega_entropy}.
\end{proof}

\subsection{Main theorem} 
We are in a position to state the main result of this paper.

\begin{thm}\label{main}
Let $G$ be amenable and let $(X,G)$ be a subshift with specification.  If $(X,G)$ contains at least one free element then $(X,G)$ is universal.
\end{thm}

Before we pass to the proof, we would like to show that (under the assumption of existence of a free element) this result covers our original goal, showing that $K$-shifts are universal. The following lemma holds for any countable group $G$.

\begin{lem}\label{specK} Let $K\subset G$ be a finite set containing $e$ and at least one more element. The $K$-shift \,$\Omega_K$ has the specification property with the margin $M=KK^{-1}K$.
\end{lem}
\begin{proof}
	We need to show that if $V_1,V_2$ are two maximal $K$-separated subsets of $G$ and $F_1, F_2\subset G$ are $KK^{-1}K$-apart then there exists a maximal $K$-separated set $V\subset G$ such that 
\begin{equation}\label{KK}
V\cap F_1=V_1\cap F_1\text{ \ and \ }V\cap F_2=V_2\cap F_2.
\end{equation}
The sets $K^{-1}KF_1$ and $K^{-1}KF_2$ are $K$-apart, hence the union
$(V_1\cap K^{-1}KF_1)\cup (V_2\cap K^{-1}KF_2)$ is a $K$-separated set. Thus,
there exists a maximal $K$-separated set $V$ such that $(V_1\cap K^{-1}KF_1)\cup(V_2\cap K^{-1}KF_2)\subset V$. Suppose that \eqref{KK} does not hold. Then either $V\cap (F_1\setminus V_1)\neq\varnothing$ or $V\cap (F_2\setminus V_2)\neq\varnothing$. By symmetry, it suffices to get a contradiction in the first case.
Let $g\in V\cap F_1$ and $g\notin V_1$. Then, by (I), there exists an element $g'\in V_1\cap K^{-1}Kg\subset V_1\cap K^{-1}KF_1$. But the latter set is a subset of $V$ which implies that $g'\in V$. Since $g\in V$, $g\neq g'$ and $g'\in K^{-1}Kg$, we have a contradiction to the property (II) of maximal $K$-separated sets.
\end{proof}

\begin{rem}\label{rem1} The seemingly more natural margin $K^{-1}K$ is insufficient. Indeed, let $G=\Z^2$ and let $K=\{(0,1),(0,0),(1,0)\}$. 
The block $\alpha_1$ over the ``triangle'' between the points $(-2,1),(-2,3),(0,3)$ with $1$'s at $(-2,1)$ and $(0,3)$ and $0$'s otherwise is $\Omega_K$-admissible. So is the block $\alpha_2$ over the ``triangle'' between the points $(1,-2),(3,-2),(3,0)$ with $1$'s at $(1,-2)$ and $(3,0)$ and zeros otherwise. Note that these triangles are $K^{-1}K$-apart. But $\alpha_1\cup\alpha_2$ is not $\Omega_K$-admissible: by (I), $\alpha_1$ forces a $1$ at $(0,1)$ and $\alpha_2$ forces a $1$ at $(1,0)$, while these two positions are not $K$-apart (see Figure~\ref{fig2}). 
\begin{figure}[ht]
\includegraphics[width=7.1cm]{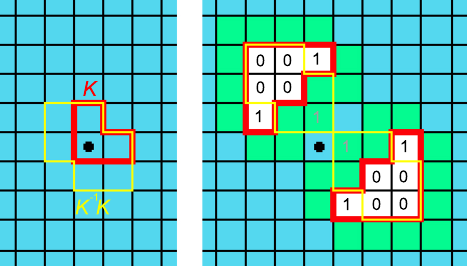}
\caption{\footnotesize The black dot shows the origin. The figure on the left shows the set $K$ (in red) and $K^{-1}K$ (in yellow). The figure on the right shows the blocks $\alpha_1$ and $\alpha_2$. The disjoint green areas show that the domains of these blocks are $K^{-1}K$ apart. The yellow frame around the point $(-1,2)$ must contain a $1$, and the only possible place is $(0,1)$ (shown in gray). Similarly, the yellow frame around $(2,-1)$ must contain a $1$ at $(1,0)$. The configuration of the gray 1's is not admitted in $\Omega_K$.}\label{fig2} 
\end{figure}
\end{rem}

\section{Tools used in the proof}

\subsection{Banach density}
\begin{defn}\label{banach}
Let $G$ be a countable group. The \emph{lower Banach density }of a subset $A\subset G$ is defined as 
	$$
	\underline d_{\mathsf B}(A)=\sup\Bigl\{\inf_{g\in G}\frac{|A\cap Fg|}{|F|}:F
	\subset G, |F|<\infty\Bigr\}.
	$$
The \emph{upper Banach density} of $A$ is defined as	
	$$
	\overline d_{\mathsf B}(A)=\inf\Bigl\{\sup_{g\in G}\frac{|A\cap Fg|}{|F|}:F
	\subset G, |F|<\infty\Bigr\}.
	$$
If $\underline d_{\mathsf B}(A)=\overline d_{\mathsf B}(A)$ then we denote the common value by $d_{\mathsf B}(A)$ and call it the \emph{Banach density} of $A$.
\end{defn}
One may easily verify that for every set $A\subset G$ it holds that $\underline d_{\mathsf B}(A)=1-\overline d_{\mathsf B}(A^c)$.

It is known (see e.g. \cite[Lemma~2.9]{DHZ}) that if $G$ is amenable and $\F=(F_n)_{n\in\N}$ is a F\o lner \sq\ in $G$ then, for any subset $A\subset G$,
\begin{equation}\label{dban_folner}
\underline d_{\mathsf B}(A)=\lim_{n\to\infty}\Bigl\{\inf_{g\in G}\frac{|A\cap F_ng|}{|F_n|}\Bigr\},
\end{equation}
and a symmetric formula holds for $\overline d_{\mathsf B}(A)$ with $\displaystyle{\sup_{g\in G}}$ in place of $\displaystyle{\inf_{g\in G}}$.

\begin{rem}\label{dr}
It is easy to see that every $K$-separated set $V$ has an upper Banach density at most $\frac1{|K|}$.
\end{rem}

\subsection{Quasitilings and dynamical quasitilings}
\begin{defn}\label{qt}
Let $G$ be a countable group and let $\S=\{S_1,S_2,\dots\}$ be a (finite or infinite) collection of finite subsets of $G$ called \emph{shapes}. By convention, each shape
contains the unity $e$. By a \emph{quasitiling} of $G$ with \emph{shapes} in $\S$ we mean any family $\T$ of subsets of $G$ that has the following structure: 
$$
\T=\{Sc:S\in\S, c\in C_S\},
$$ 
where, for each $S$, $C_S$ is a subset (usually infinite) of $G$. A set $T=Sc\in\T$ is called a \emph{tile} of shape $S$ with \emph{center} at $c$. We require that if $S\neq S'$ then $C_S\cap C_{S'}=\varnothing$. We also require that each tile has a determined shape and center, that is, whenever $Sc=S'c'$, where $S,S'\in\S$, $c\in C_S$ and $c'\in C_{S'}$ then $S=S'$ and $c=c'$. Often, given a tile $T=Sc\in\T$, we will denote its center by $c_T$. The union $\bigcup_{S\in\S}C_S$ will be called the \emph{set of centers of $\T$} and denoted by $C(\T)$.
\end{defn}

In this generality, a quasitiling can be equivalently defined as nearly any family of distinct finite subsets $T$ of $G$ called \emph{tiles}, each having a distinguished element $c_T\in T$, called the \emph{center} of $T$. The only restriction is that no two tiles share a common center (this requirement causes that not all families of distinct finite sets can play the role of quasitilings). Given a tile $T$, the set $S=Tc_T^{-1}$ is the \emph{shape} of $T$ and $\S$ is the collection of all shapes.

\begin{defn}\label{types}
A quasitiling $\T$ is called:
\begin{itemize}
	\item \emph{proper}, if $|\S|<\infty$;
	\item \emph{$(1-\eps)$-covering}, if $\bigcup\T=\bigcup_{T\in\T}T$ has lower Banach density at least $1-\eps$;
	\item \emph{disjoint}, if, for any $T,T'\in\T$, $T\neq T'\implies T\cap T'=		
	\varnothing$;
	\item \emph{$\eps$-disjoint}, if there exists a mapping $T\mapsto\tilde T$ associating to each tile $T\in\T$ its $(1-\eps)$-subset $\tilde T$ containing $c_T$, so that $T\neq T'\implies\tilde T\cap\tilde T'=\varnothing$ (the collection $\tilde \T=\{\tilde T:T\in\T\}$, where for each $T\in\T$ we set  $c_{\tilde T}=c_T$, is a disjoint quasitiling);
	\item \emph{strongly $\eps$-disjoint} if it is $\eps$-disjoint with $$\tilde T=T\setminus \bigcup\{T'\in\T:\, T\neq T',\, |T'|\ge|T|\};$$
	\item \emph{complete}, if $\bigcup\T=G$;
	\item a \emph{tiling}, if $\T$ is disjoint and complete.
\end{itemize}
\end{defn}

Quasitilings allow to quickly estimate lower Banach density of certain sets:

\begin{prop}\label{ude} Let $G$ be amenable.
Fix two small (smaller than $1$) positive numbers $\eps$ and $\delta$. Let $\T$ be any $(1-\eps)$-covering (proper) quasitiling. In each tile $T$ of $\T$ select a $(1-\delta)$-subset $T'$ so that the sets $T'$ are mutually disjoint. Then $\underline d_B(\bigcup_{T\in\T}T')\ge(1-\delta)\cdot(1-\eps)$. 
\end{prop}
\begin{proof} Let $\F=(F_n)_{n\in\N}$ be a F{\o}lner sequence in $G$. Let $U=\bigcup_{S\in\S}SS^{-1}$, where $\S$ is the collection of shapes of $\T$. Fix a $\vartheta>0$ much smaller than $\eps$. By formula~\eqref{dban_folner}, for all sufficiently large $n$ and all $g\in G$, the intersection $F_ng\cap\bigcup\T$ is a $(1-\eps-\vartheta)$-subset of $F_ng$. By Proposition~\ref{ksets}~(2) for large enough $n$, the core $(F_n)_U$ is a $(1-\vartheta)$-subset of $F_n$. For $g\in G$ let $(F_ng)_{\T}=\bigcup\{T\in\T: T\subset F_ng\}$. It is easy to verify that $(F_ng)_{\T}\supset (F_ng)_U\cap\bigcup\T$, and hence $(F_ng)_{\T}$ is a $(1-\eps-2\vartheta)$-subset of $F_ng$. Let $(F_ng)_{\T'}:=\bigcup\{T': T\in\T,\, T\subset F_ng\}$. Since the sets $T'$ are mutually disjoint and for each $T$ we have $|T'|\ge (1-\delta)|T|$, we have $|(F_ng)_{\T'}|\ge (1-\delta)|(F_ng)_{\T}|$. Therefore,
$$\frac{|(F_ng)_{\T'}|}{|F_ng|}=\frac{|(F_ng)_{\T'}|}{|(F_ng)_{\T}|}\cdot\frac{|(F_ng)_{\T}|}{|F_ng|}\ge(1-\delta)\cdot(1-\eps-2\vartheta).$$
Since $\vartheta$ is arbitrarily small, the above estimate implies $$\underline d_B(\bigcup_{T\in\T}T')\ge(1-\delta)\cdot(1-\eps).$$
\end{proof}

In particular, it follows that if a quasitiling $\T$ is $\eps$-disjoint and $(1-\eps)$-covering then the associated disjoint quasitiling $\tilde\T$ is $(1-2\eps)$-covering (in such case $\delta=\eps$ and then $(1-\delta)\cdot (1-\eps)=(1-\eps)^2>1-2\eps$).
\medskip

A quasitiling $\T$ can be represented as a symbolic element over the alphabet $\S_0=\S\cup\{0\}$, where $0$ is an additional symbol. In this setting, $\T\in\S_0^G$ is given by
$$
\T(g) = 
\begin{cases}
S;& Sg \text{ is a tile}\\
0;& \text{ otherwise}.
\end{cases}
$$ 
By a \emph{dynamical quasitiling} we will mean any subshift $\TT$ with the alphabet $\S_0$ (if $\S$ is infinite, we equip $\S$ with the discrete topology and $\S_0$ with the topology of the one-point compactification of $\S$). A dynamical quasitiling $\TT$ is called \emph{proper}, $(1-\eps)$-\emph{covering}, \emph{disjoint}, $\eps$-\emph{disjoint}, \emph{strongly $\eps$-disjoint}, \emph{complete}, or a \emph{dynamical tiling}, if all its elements $\T\in\TT$ are proper, $(1-\eps)$-covering, disjoint, $\eps$-disjoint, strongly $\eps$-disjoint, complete, or tilings, respectively.

Unless clearly stated otherwise, every quasitiling and dynamical quasitiling considered in this paper will be proper by default. The only exception is the improper dynamical tiling $\hat\TT$ introduced in Lemma~\ref{quasitofull}.

From now on, we will use the notation $\tilde\T$ exclusively for disjoint (proper) quasitilings. The tilde will also be put over letters denoting other objects associated with disjoint quasitilings (such as a disjoint dynamical quasitiling $\tilde\TT$, the set of shapes $\tilde\S$, individual shapes $\tilde S\in\tilde\S$ or individual tiles $\tilde T\in\tilde\T$).
 
\subsection{Congruent systems of disjoint quasitilings}
A \sq\ of dynamical quasitilings $(\TT_n)_{n\in\N}$ gives rise to a countable \tl\ joining\footnote{By
a \emph{topological joining} of a sequence of dynamical systems $(X_k,G)$, $k\in\N$, denoted by $\bigvee_{k\in\N}X_k$, we mean any closed subset of the Cartesian product $\prod_{k\in\N}X_k$, which has full projections onto the coordinates $X_k$, $k\in\N$, and is invariant under the product action given by $g(x_1,x_2,\dots)=(g(x_1),g(x_2),\dots)$. We note that the symbol $\bigvee_{k\in\N}X_k$ refers to many possible topological joinings.} $\TTT=\bigvee_{n\in\N}\TT_n$ called a \emph{system of quasitilings}. There are many types of such systems, but in this paper we will use only one: F\o lner, congruent systems of disjoint quasitilings. 
\begin{defn}\label{csq}
By a \emph{congruent system of disjoint quasitilings} we mean a \tl\ joining $\tilde\TTT=\bigvee_{n\in\N}\tilde\TT_n$, where, for each $n\in\N$, $\tilde\TT_n$ is a disjoint, $(1-\eps_n)$-covering dynamical quasitiling with the set of shapes $\tilde\S_n$. The elements of $\tilde\TTT$ are \sq s $(\tilde\T_n)_{n\in\N}$ with $\tilde\T_n\in\tilde\TT_n$ for each $n\in\N$. Moreover, we require that
\begin{itemize}
	\item $\lim_{n\to\infty}\eps_n=0$ ($\tilde\TTT$ with this property is called \emph{eventually $1$-covering}),
  \item for each $(\tilde\T_n)_{n\in\N}\in\tilde\TTT$ and each $n\in\N$, 
	every tile of $\tilde\T_n$ is either a subset of some 
	tile of $\tilde\T_{n+1}$ or it is disjoint from all tiles of $\tilde\T_{n+1}$ (this 
	property is called \emph{congruency} of $\tilde\TTT$).
\end{itemize}
\end{defn}

\begin{defn}\label{primary}
Let $\tilde\TTT$ be a congruent system of disjoint quasitilings. For every $(\tilde\T_n)_{n\in \N}\in\tilde\TTT$, every $n\in\N$ and every tile $\tilde T\in\tilde\T_{n+1}$, by \emph{primary subtiles} $\Breve T$ of $\tilde T$ we mean the tiles of $\tilde\T_k$ with $1\le k\le n$, contained in $\tilde T$ and not contained in any tile of $\tilde\T_{l}$ with $k<l\le n$. 
\end{defn}

\begin{defn}\label{Fsq}
Let $\TTT=\bigvee_{n\in\N}\TT_n$ be a system of quasitilings. Let $\S_n$ denote the set of shapes of $\TT_n$. We will say that $\TTT$ is \emph{F\o lner} if $\bigcup_{n\in\N}\S_n$ arranged in a sequence is a F\o lner sequence in $G$ (this definition applies only to amenable groups).
\end{defn}

\begin{defn}\label{gp}
Let $\TTT=\bigvee_{n\in\N}\TT_n$ be a system of quasitilings. A \sq\ $(\tilde\T_n)_{n\in \N}\in\tilde\TTT$ is said to be \emph{in general position} if $\bigcup_{n\in\N}\tilde T^{(n)}=G$, where $\tilde T^{(n)}$ is the ``central'' tile of $\tilde\T_n$, i.e., the tile containing~$e$.
\end{defn}

In \cite[Lemma~B.4]{DOWZ} it is proved that if $\tilde\TTT$ is a F\o lner congruent system of disjoint quasitilings then the set $\tilde\TTT_{\mathsf{GP}}$, consisting of the \sq s $(\tilde\T_n)_{n\in \N}\in\tilde\TTT$ which are in general position, has full measure for every \im\ on $\tilde\TTT$. Although that lemma deals with systems of tilings rather than quasitilings, the proof passes with almost no modification.

\subsection{Quasitiling factors of free \tl\ systems}
From now on we will always assume that the group $G$ is amenable and we will fix a F{\o}lner sequence $\F=(F_n)_{n\in\N}$ in $G$ which is 
\begin{itemize}
	\item \emph{symmetric}, i.e., for all $n\ge 1$, $F_n=F_n^{-1}$,
	\item \emph{nested}, i.e., for all $n\ge 1$, $F_n\subset F_{n+1}$, and
	\item \emph{centered}, i.e., for all $n\ge 1$, $e\in F_n$,
	\item \emph{exhaustive}, i.e., $\bigcup_{n\ge 1}F_n = G$.
\end{itemize}
The existence of such a F\o lner \sq\ in any countable amenable group is proved in \cite[Corollary~5.3]{N}.
We will need the following results from \cite{DH}. 
\begin{lem}\label{qtilingfactor}\cite[Lemma~3.4 and Corollary~3.5]{DH} 
\phantom{.} 
\begin{itemize}
	\item[(a)] There exists a universal function\footnote{In fact, $r_\eps$ can be defined as the smallest integer $r$ such that $(1-\eps)^r<\eps$, see \cite{DHZ}.} $\eps\mapsto r_\eps$ from $(0,1)$ to $\N$ such that whenever $(Y,G)$ is a free action of a countable amenable group on a zero-dimensional compact space then, for every $\eps>0$ and $n_0\in\N$, there exists a \tl\ factor map $\phi:(Y,G)\to(\TT,G)$, where $(\TT,G)$ is a dynamical $(1-\eps)$-covering, strongly $\eps$-disjoint\footnote{All three papers, \cite{OW}, \cite{DHZ} and \cite{DH}, use the notion of $\eps$-disjointness, but it can be verified that the $\eps$-disjoint quasitilings constructed in these papers are in fact strongly $\eps$-disjoint.} 
quasitiling with shapes $S_1=F_{n_1},S_2=F_{n_2},\dots,S_{r_\eps}=F_{n_{r_\eps}}$, where $n_0<n_1<n_2<\cdots<n_{r_\eps}$.\footnote{The quasitiling $(\TT,G)$ which is a factor of $(Y,G)$ can be thought of as a collection of $r_\eps$ $\eps$-disjoint towers (with clopen levels) in $Y$ whose union covers a set of measure at least $1-\eps$ for every \im. A point $y\in Y$ belongs to the base of the $i$th tower if and only if the quasitiling $\T_y\in\TT$ associated to $y$ has at $e$ the center of a tile with shape $S_i$.} 
  \item[(b)]	Moreover, there exists a $(1-\eps)$-covering disjoint dynamical quasitiling $(\tilde\TT,G)$ such that 
	\begin{enumerate}
		\item[\rm(i)] $(\tilde\TT,G)$ is a \tl\ factor of $(\TT,G)$,\footnote{The fact that $(\tilde\TT,G)$ is a \tl\ factor of $(\TT,G)$ is implicit in the last paragraph of the proof of \cite[Corollary~3.5]{DH}.} via a map $\tilde\phi$, 
		\item[\rm(ii)] for each $\T\in\TT$ and its image $\tilde\T=\tilde\phi(\T)\in\tilde\TT$, every tile $\tilde T$ of $\tilde\T$ is a $(1-\eps)$-subset of a tile $T$ of $\T$.
		\end{enumerate}
\end{itemize}
\end{lem}

The following supplemental observation is not included in \cite{DH}.

\begin{lem}\label{sepcen}Fix an $\eps\in(0,\frac13)$ and let $K$ be a finite subset of $G$. If $\TT$ is a strongly $\eps$-disjoint dynamical quasitiling whose collection of shapes $\S$ is linearly ordered by inclusion (i.e., $\S=\{S_1,S_2,\dots,S_r\}$ where $S_1\subset S_2\subset\cdots\subset S_r$), and all shapes are symmetric and $(K^{-1}K,\eps)$-\inv\ then, for each $\T\in\TT$, the set $C(\T)$ of centers of the tiles of $\T$ is $K$-separated.
\end{lem}

\begin{proof}Fix some $\T\in\TT$ and first consider two tiles of $\T$ with the same shape $S=S_i$, $i\in\{1,2,\dots,r\}$, $T=Sc$ and $T'=Sc'$, and suppose that $c,c'$ are not $K$-apart. Then $gc=g'c'$ for some $g,g'\in K$. We have $Sc'=S(g')^{-1}gc$ and $|Sc'\cap Sc|=|S(g')^{-1}g\cap S|=|g^{-1}g'S^{-1}\cap S^{-1}|$. By symmetry of $S$, the latter cardinality equals $|g^{-1}g'S\cap S|$, which, by $(K^{-1}K,\eps)$-invariance of $S$, is at least $(1-2\eps)|S|$ (see Proposition \ref{ksets} (1)), which is larger than $\eps|S|$. This is a contradiction with strong $\eps$-disjointness of $\T$. Now suppose that the tiles $T$ and $T'$ have different shapes $S\subsetneq S'$, respectively, and that their respective centers $c,c'$ are not $K$-apart. Arguing as before, we get that $|Sc\cap Sc'|>\eps|S|$. Then also $|Sc\cap S'c'|>\eps|S|$, and hence $|T\setminus T'|<(1-\eps)|T|$, which again contradicts strong $\eps$-disjointness.
\end{proof}

Using Lemma~\ref{sepcen} we can improve the disjoint dynamical quasitiling $\tilde\TT=\tilde\phi(\TT)$ so that it has two additional convenient properties. 

\begin{lem}\label{Fintildas}
Fix an $\eps\in(0,\frac13)$ and let $K\subset G$ be a finite set containing $e$. Let $\TT$ be the dynamical quasitiling described in Lemma~\ref{qtilingfactor} (a). If $n_0$ is large enough then there exists a disjoint, $(1-\eps)$-covering dynamical quasitiling $(\tilde\TT',G)$ such that
\begin{enumerate}
	\item[\rm(i)] $(\tilde\TT',G)$ is a \tl\ factor of $(\TT,G)$ via a map $\tilde\phi'$,
	\item[\rm(ii)] for each $\T\in\TT$ and its image $\tilde\T'=\tilde\phi'(\T)\in\tilde\TT'$, every tile $\tilde T'$ of $\tilde\T'$ is a $(1-2\eps)$-subset of a tile $T$ of $\T$,
	\item[\rm(iii)]$c_{\tilde T'}=c_T$, and 
	\item[\rm(iv)] $Kc_T\subset\tilde T'$.
\end{enumerate}
\end{lem}

\begin{proof}
Since the F\o lner \sq\ $\F$ is exhaustive, as soon as $n_0$ is sufficiently large, all shapes of $\TT$ contain $K$. Let $K'\supset K$ be a symmetric set so large that $\frac{|K|}{|K'|}<\frac\eps2$. If $n_0$ is large enough, all shapes of $\TT$ are $(K'^{-1}K',\eps)$-\inv, so, by Lemma~\ref{sepcen}, for each $\T\in\TT$ the set of centers $C(\T)$ is $K'$-separated. 
We will modify the disjoint quasitiling $\tilde\TT$ 
described in Lemma~\ref{qtilingfactor} (b) so that, in addition to conditions (i) and (ii), it will also satisfy conditions (iii) and (iv). 
If $n_0$ is large enough, all shapes of $\TT$ are $(K',
\delta)$-\inv\, where $\delta=\frac\eps{2|K'|^2}$, and then, by Proposition~\ref{aux}, for each $\T\in\TT$ and each tile $T$ of $\T$, we have
\begin{equation}\label{plp}
\frac{|KC(\T)\cap \tilde T|}{|T|}\le\frac{|KC(\T)\cap T|}{|T|}\le|K'|^2\delta +\frac{|K|}{|K'|}<\eps.
\end{equation}
Now we modify each tile $\tilde T$ of $\tilde\T\in\tilde\TT$ as follows:
$$
\tilde T'=(\tilde T\setminus KC(\T))\cup Kc_T,
$$
and set the center of $\tilde T'$ at $c_T$.
Note that since the tiles $\tilde T$ are pairwise disjoint and so are the sets $Kc_T$ (the set $C(\T)$ is $K'$-separated, hence also $K$-separated), the new tiles $\tilde T'$ are pairwise disjoint. We let $\tilde\T'$ be the disjoint quasitiling whose tiles are $\tilde T'$. Note that $\tilde\T'$ covers the same subset of the group as does $\tilde\T$, so $\tilde\T'$ is $(1-\eps)$-covering. The mapping $\tilde\T\mapsto\tilde\T'$ is a \tl\ factor map (because, roughly speaking, the creation of $\tilde T'$ from $\tilde T$ depends on the quasitiling $\tilde\T$ within a finite horizon around $\tilde T$) from $\tilde\TT$ to a dynamical quasitiling which we now denote by $\tilde\TT'$. The composition $\tilde\phi'$ of the above factor map with $\tilde\phi$ acts from $\TT$ to $\tilde\TT'$, so~(i) is fulfilled.

For each $\T\in\TT$ and each tile $T$ of $\T$ we have
$\tilde T\subset T$ and $Kc_T\subset T$, which implies that $\tilde T'\subset T$.  
Combining \eqref{plp} with the fact that $\tilde T$ is a $(1-\eps)$-subset of $T$, we find that $\tilde T'$ is a $(1-2\eps)$ subset of $T$ and, by definition, it has the same center as $T$, so conditions (ii) and (iii) are fulfilled. Condition (iv) is satisfied directly by the definition of~$\tilde T'$.
\end{proof}

Since $\tilde\TT'$ is just an improved version of $\tilde\TT$, to simplify the notation, from now on we will denote it by $\tilde\TT$. Likewise, the factor map $\tilde\phi'$ will be denoted by $\tilde\phi$. The set $K$ appearing in property (iv) will be specified later. 

\begin{lem}\label{lem_compatible}\cite[Lemma~3.6]{DH} 
	Let $Y$ be a compact zero-dimensional space on which $G$ acts freely by homeomorphisms. Choose $\eps_1>0$ and let $\tilde\psi_1:Y\to\tilde\TT_1$ be a topological factor map onto a disjoint $(1-\eps_1)$-covering dynamical quasitiling of $G$. For any finite set $F\subset G$ and any positive $\eps_2<\eps_1$, there exists a \tl\ factor map $\tilde\psi_2:Y\to\tilde\TT_2$, onto a disjoint, $(1-\eps_2)$-covering dynamical quasitiling, such that every shape of $\tilde\TT_2$ is $(F,\eps_2)$-invariant and for every $y\in Y$, every tile of $\tilde\psi_1(y)$ is either a subset of some tile of $\tilde\psi_2(y)$ or is disjoint from all tiles of $\tilde\psi_2(y)$.
\end{lem}

Combining Lemmas \ref{qtilingfactor} and \ref{lem_compatible} in a straightforward inductive argument, we obtain the following corollary:

\begin{cor}\label{cor_quasitiling}
Let $Y$ be a compact zero-dimensional space on which $G$ acts freely by homeomorphisms. There exists a topological factor map $\tilde\Phi:(Y,G)\to (\tilde\TTT,G)$, 
where $(\tilde\TTT,G)$ is a F\o lner, congruent system of disjoint quasitilings. We can arrange that the first ``level'' $\tilde\TT_1$ of $\tilde\TTT$ is any a priori given disjoint dynamical quasitiling which is a factor of $Y$. Using \cite[Lemma~B.4]{DOWZ}, we obtain that the set $Y^*=\tilde\Phi^{-1}(\tilde{\TTT}_{\mathsf{GP}})$ (see comments below Definition~\ref{gp}) has measure $1$ for every \im\ on~$Y$.
\end{cor}

The next lemma is preparatory for Lemma \ref{quasitofull} in which we will factor a disjoint quasitiling $\tilde\TT$ onto an improper tiling $\hat\TT$ (i.e., with possibly infinitely many shapes).

\begin{lem}\label{comparison}
Let $(Z,\eta,G)$ be an ergodic $G$-system and let $A,B\subset Z$ be disjoint measurable sets satisfying $\eta(A)<\eta(B)$. Then there exists a countable measurable partition of $A$, $A=A_0\sqcup A_1\sqcup A_2\sqcup\cdots$ with $\eta(A_0)=0$, and elements $g_1,g_2,\dots$ of $G$ such that the sets $g_n(A_n)$, $n\ge1$, are disjoint and contained in $B$.
\end{lem}
\begin{proof}Enumerate $G=\{g_0,g_1,g_2,\dots\}$ so that $g_0=e$. Let $A_1=A\cap g_1^{-1}(B)$. Next, define inductively
$$
A_n = \Bigl(A\setminus\bigcup_{i=1}^{n-1}A_i\Bigr)\ \cap \ g_n^{-1}\Bigl(B\setminus\bigcup_{i=1}^{n-1}g_i(A_i)\Bigr).
$$
It is clear that the sets $A_n$ are disjoint, contained in $A$, and that their images $g_n(A_n)$ are disjoint and contained in $B$. By preservation of the measure, the set $B_0=B\setminus\bigcup_{n\ge1}g_n(A_n)$ has positive measure, and hence is nonempty.

Suppose that the set $A_0=A\setminus\bigcup_{n\ge1}A_n$ has a positive measure. By ergodicity, there exists $n\ge1$ such that $g_n(A_0)\cap B_0$ has a positive measure. This is a contradiction, because in this case the set $g_n^{-1}(g_n(A_0)\cap B_0)$ should have been included in $A_n$, while it is contained in $A_0$, disjoint from $A_n$.
\end{proof}

\begin{lem}\label{quasitofull}Fix an $\eps\in(0,\frac13)$.
Let $(\tilde\TT,G)$ be a disjoint $(1-\eps)$-covering dynamical quasitiling of $G$ and let $\tilde\eta$ be an ergodic measure on $\tilde\TT$. Then there exists a dynamical improper tiling $(\hat{\TT},G)$ with an ergodic measure $\hat\eta$ such that:
\begin{enumerate}
	\item $(\hat\TT,\hat\eta,G)$ is a measure-theroetic factor of $(\tilde\TT,\tilde\eta,G)$ via a map $\hat\phi$,
	\item for $\tilde\eta$-almost every element $\tilde\T$ of $\tilde\TT$, the element $\hat\T=\hat\phi(\tilde\T)$ is a partition of $G$ into tiles, 
	each being a $2\eps$-enlargement of a tile of $\tilde\T$ (and with the same center).
\end{enumerate}
\end{lem}

\begin{proof}
For each $\tilde\T\in\tilde\TT$ let $A(\tilde\T)$ denote the subset of $G$ not covered by the tiles $\tilde\T$. In each shape $\tilde S$ of $\tilde\TT$ we select a subset $B_{\tilde S}$ of cardinality $\lceil 2\eps|\tilde S|\rceil$. For each $\tilde\T\in\tilde\TT$ we let $B(\tilde\T)$ denote the union of the sets $B_{\tilde S}c$, where $\tilde Sc$ are the tiles of $\tilde\T$ represented by their shapes and centers.

We define two subsets of $\tilde\T$:
$$
A=\{\tilde\T\in\tilde\TT:e\in A(\tilde\T)\} \text{ \ and \ }
B=\{\tilde\T\in\tilde\TT:e\in B(\tilde\T)\}.
$$
Since, for any $\tilde\T\in\tilde\TT$, $A(\tilde\T)$ has upper Banach density at most $\eps$ and since the tiles of $\tilde\T$ are disjoint, by Proposition~\ref{ude} $B(\tilde\T)$ has lower Banach density at least $2\eps(1-\eps)$ (which is larger than $\eps$). Hence, by \cite[Proposotion~6.10]{DZ} we have $\tilde\eta(A)\le\eps$ and $\tilde\eta(B)>\eps$. We can now use Lemma~\ref{comparison} to partition $A$ into sets $A_0,A_1,A_2,\dots$ ($A_0$ has measure zero) so that the sets $g_n(A_n)$, $n\ge1$, where $G=\{e,g_1,g_2\dots\}$, are disjoint and contained in $B$. For $\tilde\eta$-almost every $\tilde\T$ and each $a\in A(\tilde\T)$ we have $a(\tilde\T)\in A$ and then there exists $n=n(a,\tilde\T)$ such that $a(\tilde\T)\in A_n$. Lemma \ref{comparison} states that $g_n(A_n)\subset B$, which implies that $g_na(\tilde\T)\in B$, that is, $e\in B(g_na(\tilde\T))$, or, equivalently, $g_na\in B(\tilde\T)$ and thus there exists a tile $\tilde Sc$ of $\tilde\T$ such that $g_na\in B_{\tilde S}c$. We attach $a$ to the tile $\tilde Sc$. In this manner, we distribute all elements of $A(\tilde\T)$ among the tiles of $\tilde\T$ using a measurable and shift-equivariant algorithm.\footnote{For each $a,c\in G$ and each shape $\tilde S$ of $\tilde{\TT}$, 
the set $\Xi(a,c,\tilde S)$ of tilings $\tilde\T\in\tilde\TT$ such that $a\in A(\tilde\T)$, $\tilde Sc$ is a tile of $\tilde\T$, and $a$ is attached to $\tilde Sc$, is measurable. Moreover, for any $g\in G$, $g(\Xi(a,c,\tilde S))=\Xi(ag, cg, \tilde S)$.} 
The resulting collection of enlarged tiles $\hat Sc$ obtained in the process described above is an improper tiling $\hat\T$. The map $\hat\phi$ sending each $\tilde\T$ to $\hat\T$ is a measure-theoretic factor map from $\tilde\TT$ to an (improper) dynamical tiling $\hat\TT$. Note that each tile $\tilde Sc$ of $\tilde\T$ can gain at most $|B_{\tilde S}|$ (i.e., less than $\lceil2\eps|\tilde S|\rceil$) elements. Indeed, for two elements $a,a'$ of $A(\tilde\T)$, let $n=n(a,\tilde\T)$ and $n'=n(a',\tilde\T)$, and suppose that $g_na = g_{n'}a'$ (i.e., both $a$ and $a'$ are ``attached'' to the same element of $B(\tilde\T)$). Then we have $a(\tilde\T)\in A_n$ and $a'(\tilde\T)\in A_{n'}$ and, at the same time, $g_na(\tilde\T) = g_{n'}a'(\tilde\T)$, implying that $g_n(A_n)\cap g_{n'}(A_{n'})\neq\varnothing$. This is possible only when $n=n'$ and thus $a=a'$. 
\end{proof}

\subsection{Topological models of free ergodic systems}

\begin{lem}\label{modele}
Any free ergodic system $(Y,\nu,G)$ with finite entropy admits the following two \tl\ models: 
\begin{enumerate}
	\item a free \tl\ action $(Y_0,G)$ on a zero-dimensional space $Y_0$ supporting an ergodic measure $\nu_0$ isomorphic to $\nu$,
	\item a uniquely ergodic subshift $(Y',G)$ such that its unique \im\ $\nu'$ is isomorphic to $\nu$.
\end{enumerate}
\end{lem}

\begin{proof}[Proof of Lemma \ref{modele}]
The existence of the model $(Y_0,G)$ is a direct corollary of a result by G.\ Elek, \cite[Theorem~2]{E}.
The existence of the second model follows from \cite[Theorem~2']{R}, subject to the following explanations: The proof of \cite[Theorem~2']{R} is based on Rosenthal's preprint \emph{Strictly ergodic models and amenable group actions} where he proved a version of the Jewett--Krieger Theorem for amenable group actions: 
\begin{itemize} 
\item[(*)] every free ergodic action $(Y,\nu,G)$ of an amenable group has a strictly ergodic free model $(Y'',G)$. 
\end{itemize}
Since that preprint has not been published, we can instead use a recent note by B.~Weiss \cite{W}, where (*) is shown, using a 2018 result of B.\ Frej and D.\ Huczek \cite[Theorem~1.2]{FH}. 
From here we assume that $(Y,\nu,G)$ has finite entropy. In this case, the model $(Y'',\nu'',G)$ also has a finite \tl\ entropy. Then \cite[Theorem~2']{R} asserts that there exists a uniform generator $\P$ in $Y''$ with a finite number of atoms (in fact, $|\P|=\lfloor\log{h(\nu)}\rfloor+1$). Skipping the definition of a uniform generator (which can be found in \cite{R}), it suffices to know that uniformity of $\P$ implies that all points in the subshift $(Y',G)$, defined as the closure of the $\P$-names of points belonging to a certain subset $Y'''\subset Y''$ with $\nu''(Y''')=1$, are generic for a measure $\nu'$ isomorphic to $\nu$. So, the subshift $(Y',G)$ is a uniquely ergodic symbolic model for $(Y,\nu,G)$, as claimed in (2).
\end{proof}

\begin{rem}
The symbolic model $(Y',G)$ need not be \tl ly free. This is why we need both models. 
It is quite possible that there exists a free uniquely ergodic symbolic model, but we have neither citation nor proof to support such a claim.
\end{rem}

From now on we will identify the measure-theoretic ergodic system $(Y,\nu,G)$ with its model $(Y_0,\nu_0,G)$ from Lemma \ref{modele} (1), that is, each time we write $(Y,\nu,G)$ we will assume that $(Y,G)$ is a free \tl\ action of $G$ on a compact zero-dimensional space $Y$ and $\nu$ is one of the ergodic measures on $Y$.

We let $\pi:Y\to Y'$ denote the isomorphism between $(Y,\nu,G)$ and the uniquely ergodic subshift $(Y',\nu',G)$ from Lemma \ref{modele} (2).

\section{Sketch of the proof of the main theorem} 
\subsection{Preliminary reshaping}\label{skecz} Let $(X,G)$ be a subshift with the specification property, containing a free element, and let $M$ denote the margin of specification. Let $(Y,\nu,G)$ be a free ergodic measure-preserving system such that $h(\nu)<\htop(X)$. Our goal is to construct a measure-theoretic isomorphism $\Psi:(Y,\nu,G)\to (X,\mu,G)$. For $\nu$-almost every $y\in Y$ we need to create a symbolic element $x=\Psi(y)\in X$, so that $\Psi$ is measurable, $\nu$-almost everywhere invertible and shift-equivariant. Then the image of $\nu$ via $\Psi$ is going to be the desired measure $\mu$ on $X$, isomorphic to $\nu$.

By convention, $(Y,G)$ is a free \tl\ zero-dimensional system and $\nu$ is one of the ergodic measures in $Y$. Also, let $\Delta$ denote the alphabet of the other (symbolic and uniquely ergodic) model $(Y',G)$. By the variational principle, $\htop(Y')=h(\nu)$.

By Lemma~\ref{qtilingfactor}~(a), there exists a \tl\ factor map $\phi:(Y,G)\to\TT$ to a strongly $\eps$-disjoint, $(1-\eps)$-covering dynamical quasitiling having $r_\eps$ shapes $S_1=F_{n_1},S_2=F_{n_2},\dots,S_{r_\eps}=F_{n_{r_\eps}}$, where $n_0<n_1<n_2<\cdots<n_{r_\eps}$. Furthermore, by Lemma~\ref{Fintildas} (which is an improved version of Lemma~\ref{qtilingfactor}~(b)), each quasitiling $\T\in\TT$ determines (via a \tl\ factor map $\tilde\phi$) a disjoint quasitiling $\tilde\T\in\tilde\TT$ whose shapes are $(1-2\eps)$-subsets of those of $\T$ and have the same centers (and satisfy (iv) for some set $K$ whose role will be revealed in a moment). Finally, by Lemma~\ref{quasitofull}, $\tilde\T$ determines (via a measure-theoretic factor map $\hat\phi$) an improper tiling $\hat\T\in\hat\TT$ whose shapes are $2\eps$-enlargements of those of $\tilde\T$ and still have the same centers. 

The parameters $\eps$ and $n_0$ decide about the covering property of $\tilde\TT$ and the invariance property of its shapes. We will establish these parameters later. 

\subsection{Sketch of the construction of $\Psi$.}
We start by giving a sketch in which we skip most of the technical details. These will be given later, in the formal part of the proof. We hope that the sketch will help the reader keep track of the logic behind the formal proof.

The general idea is quite simple. Using the entropy gap between the subshifts $X$ and $Y'$ we will create an injective code associating to every block of the form $\alpha\approx y'|_{\hat T}$, where $\hat T=\hat Sc$ is a tile of the (improper) tiling $\hat\T=\hat\phi\circ\tilde\phi\circ\phi(y)$ associated to $y=\pi^{-1}(y')$, an $X$-admissible block $x|_{\tilde T}\approx\theta_{\hat S,\tilde S}(\alpha)$ over the tile $\tilde T=\tilde Sc$ of $\tilde\T$ contained in $\hat T$. The entropy gap allows us to compress the code a bit, so that the domain of $\theta_{\hat S,\tilde S}(\alpha)$ is slightly smaller than $\tilde T$; we leave a small margin $\tilde T\setminus\tilde T_{\bar M}$ of $\tilde T$ (the set $\bar M$ will be specified later) and a small ``hole'' $Kc$ in the center temporarily empty (this is where the property (iv) is needed). The margin is necessary for future ``gluing'' the blocks $\theta_{\hat S,\tilde S}(\alpha)$ using the specification property, the role of the hole $Kc$ will be explained in a moment. The ``code'' $\theta_{\hat S,\tilde S}$ is injective, but depends on the shapes $\hat S$ and $\tilde S$. Thus, in order to be able to invert the mapping $\Psi$, i.e., recover $y'$ from $x$, we need to ``memorize'' in $x$ the center $c$ and the shapes $\hat S$ and~$\tilde S$. 

This is the role of the central hole. We put there in $x$ (also leaving within the hole a small margin for ``gluing'') a special block, called the ``marker'', which indicates the center $c$ and ``memorizes'', instead of the shapes $\hat S$ and $\tilde S$, the shape $S$ of the tile $T=Sc$ of the initial quasitiling $\T=\phi(y)$. This is an essential simplification, since $\T$ has relatively few shapes $S_1,S_2,\dots,S_{r_\eps}$, while the number of shapes $\tilde S$ is unknown, and, even worse, there may be infinitely many shapes $\hat S$, so that memorizing them in a finite marker would be impossible. So, why is memorizing the centers and shapes of $\T$ sufficient? First of all, $T$, $\hat T$ and $\tilde T$ share a common center, thus the question only concerns the shapes. The answer is as follows. Once we identify {\bf all} markers occurring in $x$, we will be able to recover the entire quasitiling $\T$. Since both $\tilde\TT$ and $\hat\TT$ are factors of $\TT$, $\T$ {\bf determines} both quasitilings $\tilde\T$ and $\hat\T$, so we will also know the shapes $\tilde S$ and $\hat S$ of any tiles of these quasitilings (since $\hat\phi$ is a measure-theoretic factor map, in order to recover $\hat\T$ from $\T$ we really need to see the entire quasitiling $\T$; no finite horizon viewing is sufficient).

Markers must be ``unambiguous'', that is, they must occur in the elements $x=\Psi(y)$ {\bf exclusively} at the centers of the tiles. We will devote a long section to building the markers in a way that eliminates their unintentional occurrences.

Finally, we need to overcome one last difficulty. So far, the images $x=\Psi(y)$ are filled with symbols everywhere except between the tiles $\tilde T$, over the margins of these tiles, and over the margins around the markers. The specification property allows us to fill in the missing symbols {\bf individually} for each $x$, and usually in many different ways. However, we need to do it in a way that is {\bf determined} by $y$ (so that $x=\Psi(y)$ is unique) and the dependence on $y$ must be {\bf measurable} and {\bf shift-equivariant}. This is the subject of another delicate arrangement described in subsection~\ref{killer}.
\smallskip

This sketch is largely oversimplified, but it outlines the main principles behind the construction.

\section{Markers}
In this section, we will construct the markers $\zeta_t$ (from now on, $t$ always ranges from $1$ to $r_\eps$) and restrict the range of $\Psi$ to such elements $x$ in which the markers do not occur ``by chance'', only where we place them. The construction begins with the creation of a special subsystem $\bar X$ of $X$ with an entropy almost as large as $\htop(X)$.

\subsection{Construction of the subshift $\bar X$}\label{s5.1}
Recall that $M$ (containing $e$) is the margin of specification of the subshift $X$. Since any set containing $M$ also is a margin of specification for $X$, we can assume that $M$ is symmetric (this will allow us to write $M^2$ rather than the lengthy term $M^{-1}M$).
\begin{defn}\label{ins}
Let $\alpha$ be an $X$-admissible block with a finite domain $A$ containing~$e$. Choose a symbolic element $x\in X$ and an element $c\in G$. Since the sets $Ac$ and $G\setminus M^2Ac$ are $M$-apart, there exists an element $x_\alpha\in X$ such that
\begin{enumerate}
	\item $x_\alpha|_{G\setminus M^2Ac}=x|_{G\setminus M^2Ac}$, 
	and
	\item $x_\alpha|_{Ac}\approx \alpha$.
\end{enumerate}
We will say that $x_\alpha$ \emph{is obtained by inserting $\alpha$ centered at $c$ into $x$.}
The pattern $x_\alpha|_{M^2Ac\setminus Ac}$ will be called the \emph{glue}.
\end{defn}
Note that in general the glue (and hence $x_\alpha$) is not uniquely determined by the triple $(x,c,\alpha)$, so $x_\alpha$ will denote one choice.
\smallskip

By Fact~\ref{nomin}, $(X,G)$ is not minimal, therefore there exists a proper closed \inv\ set $X'\subset X$. Then, we can choose a block $\rho$, over some finite symmetric set $R\ni e$, which is $X$-admissible but not \mbox{$X'$-admissible}. 
We fix a finite symmetric set $W$ such that
\begin{equation}\label{H3}
W\supset MR.
\end{equation}
An additional requirement as to how large $W$ must be will be specified in a moment.

If $V$ is a maximal $W$-separated set, by a \emph{grid}, denoted $\rho_V$ we will understand the pattern over $RV$ such that, for each $v\in V$,
$$
\rho_V|_{Rv}\approx\rho.
$$
Since $W\supset MR$, the domains of the occurrences of $\rho$ in a grid are $M$-apart and since $\rho$ is $X$-admissible, so is the grid. We define $\bar X\subset X$ as the set of these elements $\bar x\in X$ which contain a grid, i.e., for which there exists a maximal $W$-separated set $V$ such that $x|_{RV}=\rho_V$.
As any grid is $X$-admissible, $\bar X$ is nonempty. Since the set $\Omega_W$ consisting of the indicator functions of all maximal $W$-separated sets is closed and shift-\inv, so is $\bar X$. As $\rho$ is forbidden in $X'$, $\bar X$ is disjoint from $X'$, therefore, $(\bar X,G)$ is a proper subsystem of $(X,G)$. An element $\bar x$ of $\bar X$, can be created from any element $x\in X$ by first choosing a maximal $W$-separated set $V$, enumerating its elements as $v_1,v_2,\dots$ and then successively inserting into $x$ copies of the block $\rho$ centered at the elements $v_n$, $n\ge1$. Observe that if $\bar x$ is created from $x$ in this manner, then $\bar x$ differs from $x$ on a subset of $M^2RV$ and hence on a set of upper Banach density at most $\frac{|M^2R|}{|W|}$ (see Remark~\ref{dr}). This easily implies that 
$$
\htop(\bar X)\ge\htop(X)-\frac{|M^2R|}{|W|}\log|\Lambda|-\mathsf{H}\Bigl(\tfrac{|M^2R|}{|W|}\Bigr),
$$ 
where $\mathsf{H}(\epsilon)=-\epsilon\log{\epsilon}-(1-\epsilon)\log{(1-\epsilon)}$. The last term is subtracted to compensate for the entropy associated with an unknown partition of $G$ into two sets, one of upper density $\tfrac{|M^2R|}{|W|}$, and its complement.
By choosing $W$ large enough, we can ensure that $\htop(\bar X)>h(\nu)$. 

\subsection{The basic marker}
Let $B$ be some symmetric set containing $RW^2$. We fix an $X'$-admissible block $\beta$ over $B$ and call it the \emph{basic marker}. 

\begin{prop}\label{ban}
The basic marker does not occur in $\bar X$.
\end{prop}
\begin{proof}
Suppose $\bar x|_{Bg}\approx\beta$ for some $\bar x\in\bar X$ and $g\in G$. By the definition of $\bar X$, $\bar x$ contains a grid $\rho_V$, where $V$ is a maximal $W$-separated set. By property (I) of the maximal $W$-separated sets, the intersection $V\cap W^2g$ is nonempty. Let $v$ belong to this intersection. Then $Rv\subset RW^2g\subset Bg$ which implies that $\rho$ occurs as a subblock of $\beta$. This is a contradiction, because $\beta$ is $X'$-admissible while $\rho$ is not. 
\end{proof}

\subsection{Enhanced pattern}\label{enh}
Let $\alpha$ be an $X$-admissible pattern with a finite domain~$A$. 
After inserting $\alpha$ into a symbolic element $\bar x\in\bar X$, centered at $e$, the basic marker $\beta$ may occur in the resulting element $\bar x_\alpha\in X$, centered at some elements $c'$. Since $\beta$ is not $\bar X$-admissible, the domain $Bc'$ of such an occurrence must intersect $M^2A$, and hence $c'\in BM^2A$.
The elements $c'$ (if they exist) will be called \emph{basic centers}. The existence and positions of the basic centers depend not only on $\alpha$ but also on the ``context'' of $\bar x$ near $e$. We will now describe a procedure designed to eliminate the dependence on the context. 

We define \emph{enhanced pattern} $\underline\alpha$, as follows: We fix an element $\bar x^*\in\bar X$ (we will use the same $\bar x^*$ regardless of $\alpha$) and insert the pattern $\alpha$ into $\bar x^*$, centered at $e$. We select one element of $X$ obtained in this manner and denote it by $\bar x^*_\alpha$ (the selection is necessary, as the glue does not need to be unique). We denote by $J_\alpha$ the (possibly empty) set of centers of all occurrences of $\beta$ in $\bar x^*_\alpha$ (the set of basic centers). We have 
\begin{equation}\label{Jalfa}
  J_\alpha\subset BM^2A
\end{equation}
 and all occurrences of $\beta$ in $\bar x^*_\alpha$ fit within the set 
\begin{equation}\label{W6}
\underline A=B^2M^2A. 
\end{equation}
We let 
$$
\underline\alpha=\bar x^*_\alpha|_{\underline A}
$$
and we call $\underline\alpha$ the \emph{enhanced pattern} (associated with $\alpha$). Notice that $J_{\alpha}$ equals the set of centers of the occurrences of $\beta$ in $\underline\alpha$. The construction of $\underline\alpha$ is illustrated in Figure~\ref{enha}.
\begin{figure}[h]
\includegraphics[width=\columnwidth]{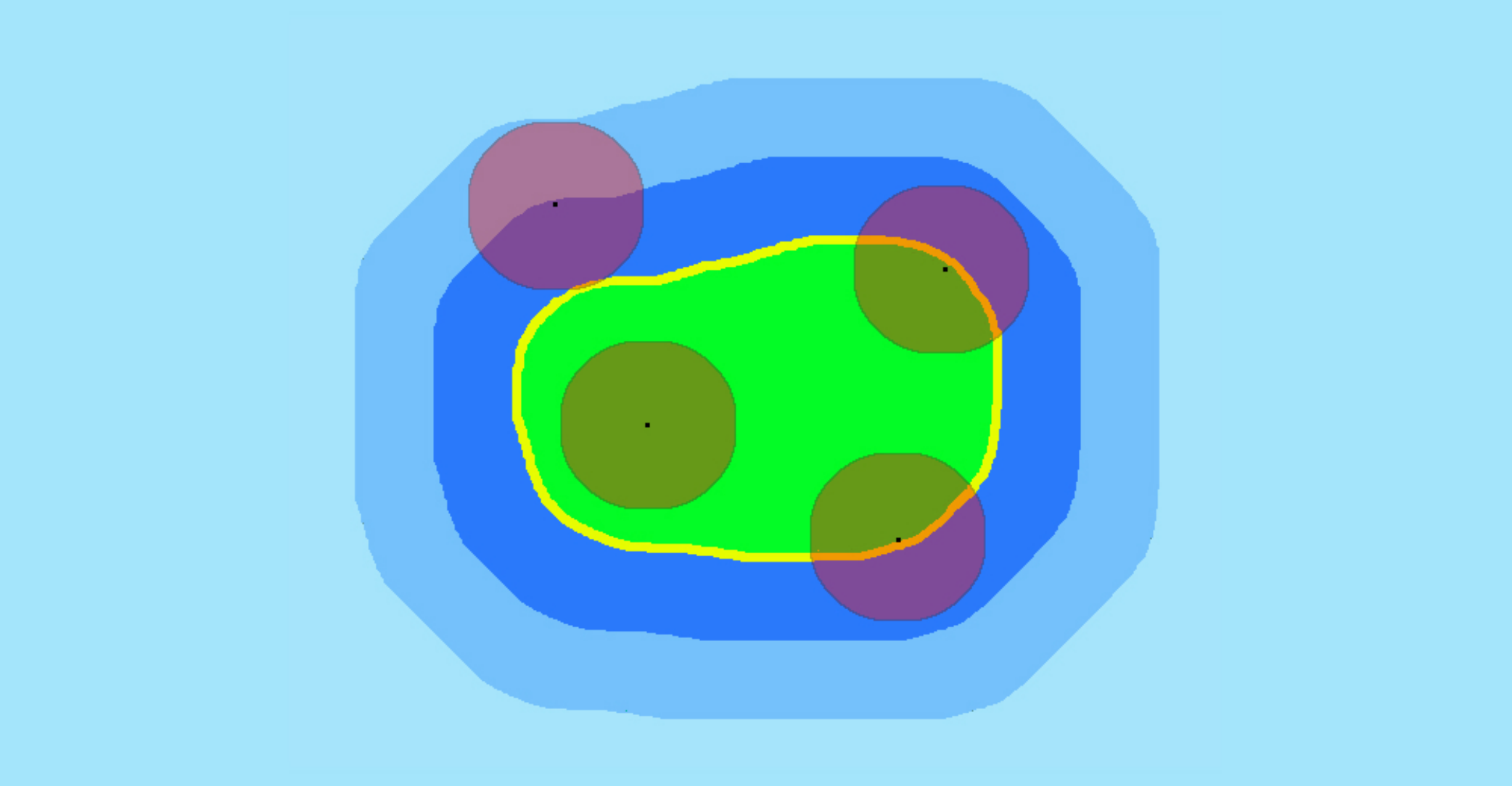}
\caption{\footnotesize An $X$-admissible pattern $\alpha$ with domain $A$ is shown in the bright green color. It is inserted into $\bar x^*$ (all shades of blue) with an $X$-admissible glue (the yellow area). The large circles are the occurrences of the basic marker $\beta$ in the resulting element $\bar x^*_\alpha$. Their domains must intersect $M^2A$ and thus their centers (black dots) fall within $BM^2A$ (green plus yellow plus dark blue area) and their domains are entirely contained within $\underline A=B^2M^2A$ (green plus yellow plus dark blue plus light blue area). The enhanced pattern $\underline\alpha$ is $\bar x^*_\alpha|_{\underline A}$.} 
\label{enha}
\end{figure}

\subsection{Permutability of $J_\beta $}
An important set is $J_\beta $, the set of basic centers in the enhanced pattern $\underline\beta$ associated with $\beta$. It consists of $\{e\}$ (the ``true center'') and perhaps some other elements belonging to $BM^2B$.

There are two essentially different possibilities for this set (see Figure~\ref{betta}): 
\begin{enumerate}
	\item $J_\beta  h\neq J_\beta $ for any $h\in G,\ h\neq e$,
	\item there exists an element $h\in G,\ h\neq e$, such that $J_\beta  h=J_\beta $ (we will say that $J_\beta $ is \emph{permutable}).
\end{enumerate}

The set $H_{\beta}=\{h\in G:J_\beta  h=J_\beta \}$ is a finite subgroup of $G$ contained in $J_\beta$, and $J_\beta $ is permutable if and only if $H_{\beta}$ is nontrivial. In particular, if $G$ is torsion-free, we always have the first case. 

In case (1) we will say that $\underline\beta$ is an ``unambiguous marker''.
This case is easier and will reappear in our discussion later, in subsection~\ref{unam}. 

In case (2) the enhanced basic marker $\underline\beta$ may be ``ambiguous'', that is, it may indicate more than one location of the center. In this case we need to build a more complicated ``unambiguous marker''. The construction utilizes the existence of a free element $x_0\in X$ and it occupies the following two subsections.

\begin{figure}[h]
\includegraphics[width=0.49\columnwidth]{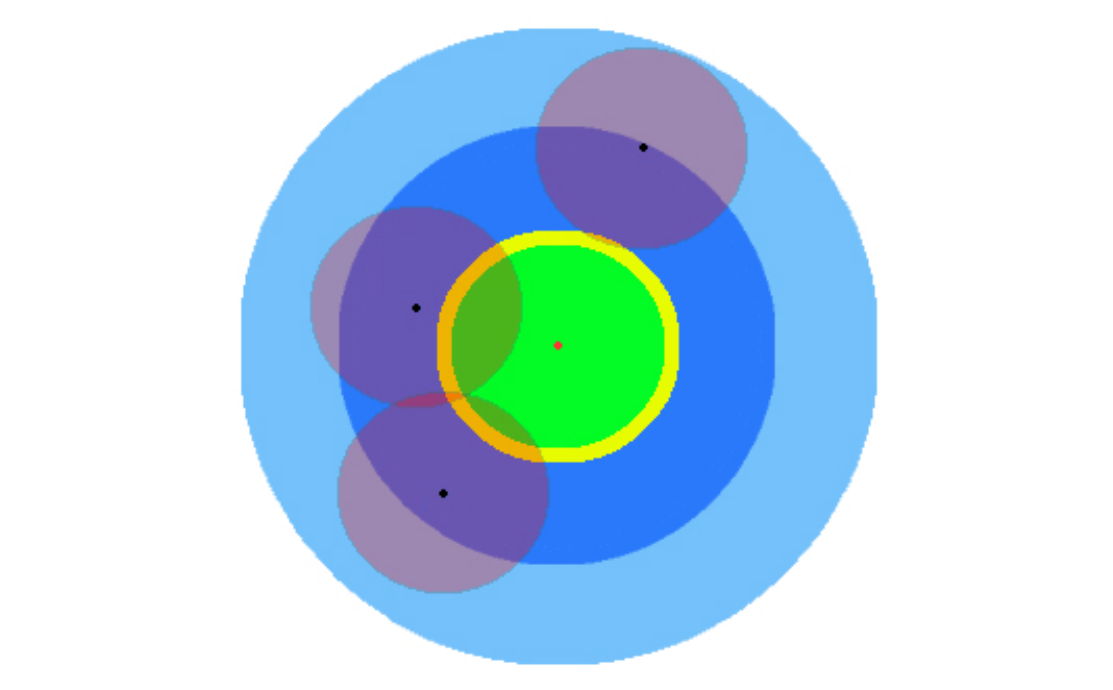}
\includegraphics[width=0.49\columnwidth]{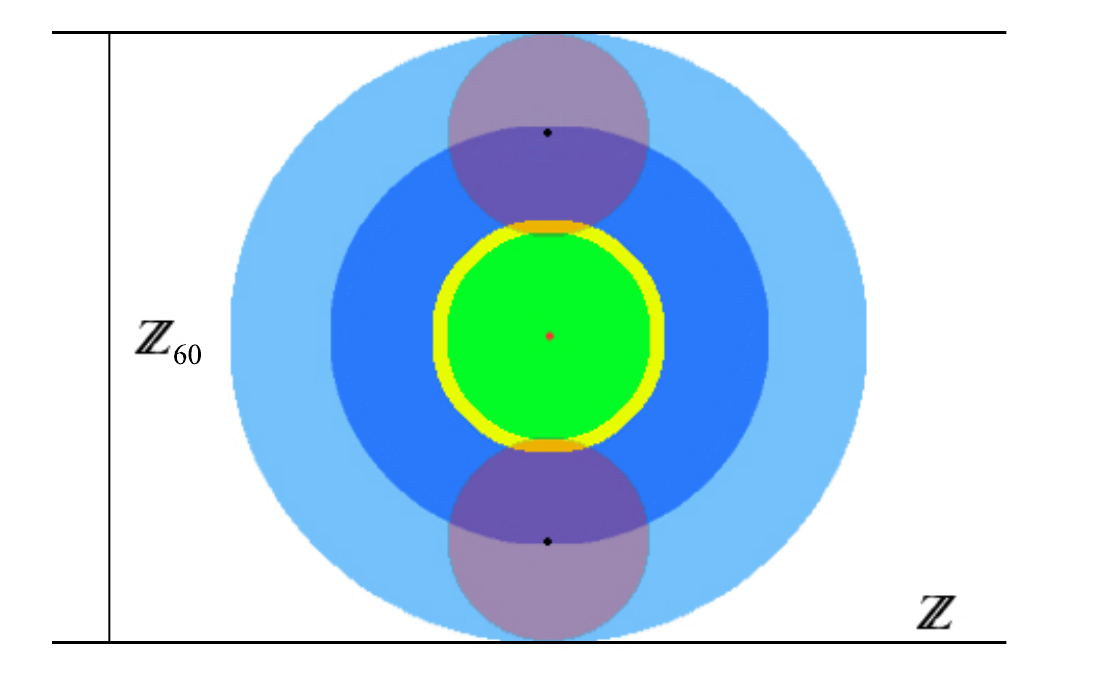}
\caption{\footnotesize The block $\underline\beta$ with $J_\beta$ nonpermutable (on the left) and permutable (on the right; here the group is $\Z\times\Z_{60}$)}
\label{betta}
\end{figure}

\subsection{Resistant block}\label{resb}
If $J_\beta $ is permutable, we can try to create an ``unambiguous marker'' by combining $\underline\beta$ with another enhanced block, say $\underline\gamma$ over a set $\underline\Gamma$, which is ``resistant'' to shifts by elements in the finite subgroup $H_{\beta}$, as specified in the definition below.

\begin{defn}\label{resis}
Let $\gamma$ be a block with a finite domain $\Gamma$ and let $H\subset G$ be a finite set. 
We say that $\gamma$ is \emph{resistant to the shifts by the elements of $H$} if, for every $h\in H\setminus\{e\}$, there exists $g\in\Gamma$ such that $gh\in\Gamma$ and $\gamma(g)\neq\gamma(gh)$. 
\end{defn}

\begin{lem}\label{res1}
If a block $\gamma$ is resistant to the shifts by the elements of a finite set $H$ and $\gamma$ occurs in some $x\in \Lambda^G$ centered at some $c$, then,  for any $h\in H\setminus\{e\}$, $\gamma$ does not occur in $x$ centered at $hc$.
\end{lem}

\begin{proof}
Suppose $\gamma$ occurs in $x$ centered at both $c$ and $hc$, where $h\in H\setminus\{e\}$. Then, by the resistance of $\gamma$, there exists $g\in\Gamma$ such that $gh\in\Gamma$ and $\gamma(g)\neq\gamma(gh)$. We have
\begin{align*}
x(ghc)=\ &\gamma(g), \text{\ since $\gamma$ occurs in $x$ centered at $hc$, while on the other hand,}\\
x(ghc)=\ &\gamma(gh), \text{\ since $\gamma$ occurs in $x$ centered at $c$},
\end{align*}
which is a contradiction with $\gamma(g)\neq\gamma(gh)$.
\end{proof}

Let $x_0$ be a free element of $X$ and let $H\subset G$ be a finite set. For each $h\in H\setminus\{e\}$ we have $x_0\neq h(x_0)$, i.e., there exists an element $g_h\in G$ such that $(h(x_0))(g_h)={x_0}(g_hh)\neq x_0(g_h)$. We let 
$$
\Gamma=\{e\}\cup\bigcup_{h\in H\setminus\{e\}}\{g_h,g_hh\}
$$ 
and then we define $\gamma = x_0|_{\Gamma}$ (the singleton $\{e\}$ is added to $\Gamma$ just to fulfill the criterion that the domain of a block should contain the unity). It is clear that $\gamma$ is $X$-admissible and resistant to the shifts by the elements of $H$, and so is any $X$-admissible block that contains $\gamma$ as a subblock centered at $e$,\footnote{A shifted copy of $\gamma$ centered at some $g\neq e$ is resistant to the shifts by the elements of $g^{-1}Hg$ rather than $H$.} in particular, so is the enhanced block $\underline\gamma$. Our idea is to create the unambiguous marker as some kind of ``union'' of the enhanced blocks $\underline\beta$ and $\underline\gamma$, where $\gamma$ is designed to be resistant to the shifts by $H_\beta$.

\subsection{Problem with shifting the blocks $\underline\gamma$ or $\underline\beta$ and a combinatorial lemma}
There emerges an additional problem. The domains of the enhanced blocks, $\underline B$ of $\underline\beta$ and $\underline\Gamma$ of $\underline\gamma$, are not disjoint. So, we need to shift either $\underline\beta$ or $\underline\gamma$ away from the origin. Suppose that we choose to shift $\underline\gamma$ and center it at some $g\in G$ (rather than at $e$), so that $\underline B$ and $\underline\Gamma g$ are sufficiently distant. Then, as easily verified, the shifted block $\underline\gamma$ is resistant to the shifts by $g^{-1}H_\beta g$, which, in the noncommutative case, may be very different from $H_\beta$. So, this idea does not seem to work. On the other hand, if we shift $\underline\beta$ and center it at some $g\in G$ then the set of centers of the occurrences of $\beta$ in the shifted block $\underline\beta$ is $J_\beta g$, rather than $J_\beta$, and the subgroup of elements $h$ satisfying $J_\beta g=J_\beta gh$ equals $g^{-1}H_\beta g$, rather than $H_\beta$. So, the block $\gamma$ designed to be resistant to the shifts by the elements of $H_\beta$ is no longer adequate. If we choose a new block, say $\gamma'$, resistant to the shifts by the elements of $g^{-1}H_\beta g$, then the previously selected ``distance'' $g$ of shifting $\underline\beta$ may turn out insufficient, and we seem to be running in circles. To deal with this problem, we need to determine a finite group $H_0$ and create the block $\gamma$ resistant to the shifts by the elements of $H_0$, {\bf independently} of the shifting distance $g$. This seemingly impossible task can be solved with the aid of the lemma given below, by combining $\underline\gamma$ with {\bf two copies} of the enhanced basic marker $\underline\beta$ centered at carefully selected elements $g_1$ and $g_2$. This purely combinatorial lemma was inspired by part (2) of the proof of \cite[Theorem~4.1]{HK} and is of independent interest.

\begin{lem}\label{kopacz} Let $(A_n)_{n\ge1}$ be any sequence of finite subsets of some space $G$, with cardinalities bounded by a finite number $M$. Then there exists a set $A\subset G$ and an infinite set $S\subset\N$ such that for any $m,n\in S, m\neq n$, we have 
\begin{equation}\label{kop}
A_m\cap A_n=A.
\end{equation}
\end{lem}
\begin{proof}
First, since some cardinality $|A_n|$ is repeated for infinitely many $n$'s,
it suffices to prove the lemma assuming that $|A_n| = M$ for all $n\in\N$. The proof
goes by induction on $M$. The lemma holds trivially (with $S=\N$ and $A=\varnothing$) if $M = 0$. Suppose that for some $M_0\ge0$ the lemma holds for all $M\in[0,M_0]$ and let $(A_n)_{n\ge1}$ be a sequence of sets of cardinality $M_0+1$. 
For each $n\ge1$ choose one element $a_n\in A_n$ and let $A'_n=A_n\setminus\{a_n\}$. 
Since for all $n\ge1$, $|A'_n|=M_0$, the inductive hypothesis implies that there exist an infinite set $S'\subset\N$ and a finite set $A'\subset G$ such that \eqref{kop} is satisfied with the sets $A'_n$, $S'$ and $A'$ in place of the sets $A_n$, $S$ and $A$, respectively.
\begin{enumerate}
	\item If $A'\neq\varnothing$, we define $A_n''=A_n\setminus A'$. 
	For each $n\in S'$, $|A''_n|\le M_0$, so by the inductive hypothesis again, there exist an infinite subset $S''\subset S'$ and a finite set $A''$ such that \eqref{kop} is satisfied with the sets $A''_n$, $S''$ and $A''$ in place of the sets $A_n$, $S$ and $A$, respectively. Then \eqref{kop} is satisfied with $S=S''$ and $A=A'\cup A''$.	
	\item If $A'=\varnothing$, we consider two subcases:
\begin{enumerate}
	\item Suppose there is an infinite subset $S''\subset S'$ and $a\in G$ such that $a_n = a$ for all $n\in S''$. Then \eqref{kop} is satisfied with the sets $S=S''$ and $A=\{a\}$.
	\item Suppose that there is no $a$ as above. Assume that some $a\in G$ is contained in infinitely many sets $A_n$ with $n\in S'$. Since $a$ may equal $a_n$ for at most finitely many $n\in S'$, it follows that $a\in A_n'$ for infinitely many $n\in S'$. However, this is impossible as $A'=\varnothing$. We have shown that in case (b), any $a\in G$ is contained in finitely many sets $A_n$ with $n\in S'$. Then one can easily select an infinite subset $S''\subset S'$ such that the sets $A_n$ with $n\in S''$ are pairwise disjoint, and \eqref{kop} is satisfied with $S=S''$ and $A=\varnothing$.  
\end{enumerate}
\end{enumerate}
The proof is now complete.
\end{proof}

\subsection{Construction of the unambiguous markers}\label{unam}
In case (1), when $J_\beta$ is nonpermutable, we let $Z=\underline B$ and define \emph{unambiguous marker} as the block over the domain $Z$, very simply, as
$$
\zeta = \underline\beta.
$$

We now return to the case (2) (of permutable $J_\beta$) and describe the construction that solves the problem described earlier.

Consider the (countable) family of finite sets (subgroups) $H_g=g^{-1}H_{\underline\beta}g$, $g\in G$. By Lemma~\ref{kopacz}, there exists an infinite set $S\subset G$ and a finite set $H_0\subset G$ (which is in fact a finite subgroup) such that for any $g_1,g_2\in S$, $g_1\neq g_2$, we have
$$
H_{g_1}\cap H_{g_2}=H_0.
$$
As shown in subsection~\ref{resb}, there exists an $X$-admissible block $\gamma$ over a finite set $\Gamma\ni e$, resistant to the shifts by the elements of $H_0$. For technical reasons (which will be revealed in a moment), we need to ensure that $|J_\gamma |>|J_\beta|$. In case $|J_\gamma |\le|J_\beta|$, we define $\gamma'$ as $\gamma$ combined with $|J_\beta|+1$ copies of $\beta$ whose domains are $M$-apart from each other and from $\Gamma$.  Clearly,  $\gamma'$ is $X$-admissible and $|J_{\gamma'}|>|J_\beta|$. Then, we redefine $\gamma$ as $\gamma'$ and $\Gamma$ as $\Gamma'$. In any case, the enhanced block $\underline\gamma$ (whose domain is~$\underline\Gamma$) is resistant to the shifts by the elements of $H_0$, because it contains $\gamma$ as a subblock centered at $e$.

Now, we define 
\begin{equation}\label{barm}
\bar M=W^3RM^2
\end{equation}
and from the infinite set $S$ we select two elements, $g_1,g_2$, so that:
\begin{align}
&\text{the sets $\underline\Gamma$, $\underline Bg_1$ and $\underline Bg_2$ are $\bar M$-apart},\\
&g_1\notin J_\beta^{-1}J_\gamma J_\gamma^{-1}J_\gamma,\label{g1}\\
&g_2\notin J_\beta^{-1}J_1J_1^{-1}J_1,\label{g2}
\end{align}
where $J_1=J_\beta g_1\cup J_\gamma$. We are in a position to define the domain $Z$ of the unambiguous marker as follows:
\begin{equation}\label{domzeta}
Z=\underline\Gamma\cup\underline Bg_1\cup\underline Bg_2=B^2M^2(\Gamma\cup B\{g_1,g_2\}),
\end{equation}
and the \emph{unambiguous marker} as the block $\zeta$ over $Z$ satisfying
$$
\zeta|_{\underline\Gamma}=\underline\gamma, \ \ \zeta|_{\underline Bg_1}\approx\underline\beta, \ \ \zeta|_{\underline Bg_2}\approx\underline\beta.
$$

The set $J_0$ of centers of occurrences of the basic marker $\beta$ within $\zeta$ equals:
\begin{align}\label{jzeta1}
J_0&=J_\beta, \text { \ in case (1)},\\
J_0&=J_\gamma\cup J_\beta g_1\cup J_\beta g_2=J_1\cup J_\beta g_2, \text { \ in case (2)}.\label{jzeta2}
\end{align}
The unambiguous marker $\zeta$ and the set $J_0$ are shown on Figure~\ref{zetat}, just ignore the small orange circles labeled~$\underline\kappa_t$; their role will be explained in a moment.

\begin{figure}[h]
\includegraphics[width=\columnwidth]{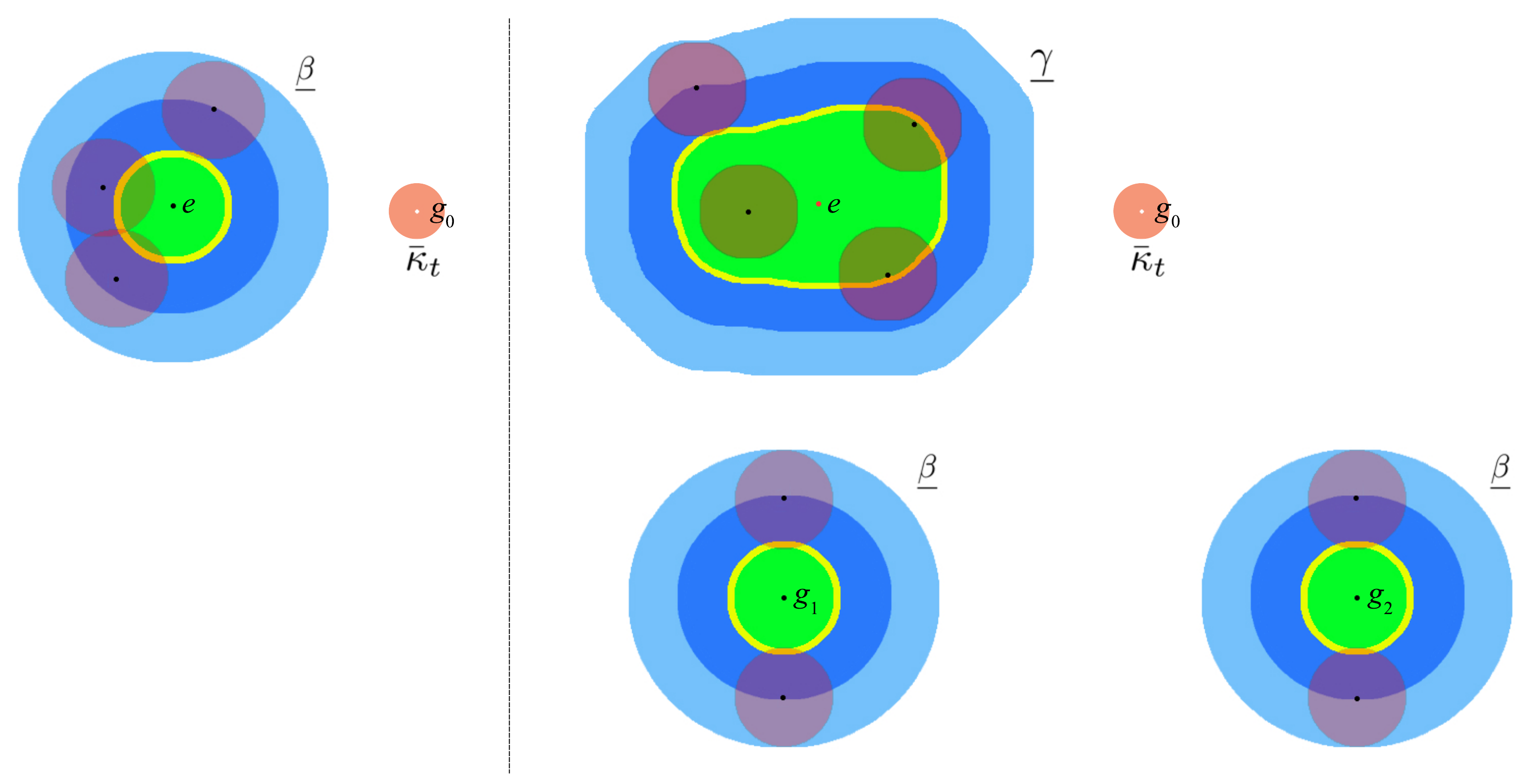}
\caption{The marker $\zeta_t$. On the left -- in case (1) the marker consists of two blocks: $\underline\beta$ centered at $e$ and $\bar\kappa_t$ centered at $g_0$. The set $J_\beta$ is shown as the four black dots including~$e$. On the right -- in case (2) the marker consists of four blocks: $\underline\gamma$ centered at $e$, two copies of $\underline\beta$ centered at $g_1$ and $g_2$, and $\bar\kappa_t$ centered at $g_0$. The set $J_\gamma$ consists of ten black dots (this time not including $e$): four within $\underline\gamma$ (this is $J_\gamma$), three within the copy of~$\underline\beta$ centered at $g_1$ (jointly these seven dots constitute $J_1$) and three dots within the copy of $\beta$ centered at $g_2$.}
\label{zetat}
\end{figure}

\begin{prop}\label{unamb}
The block $\zeta$ is \emph{unambiguous}, that is, either $J_0$ is nonpermutable, or $\zeta$ is resistant to the shifts by all elements $h\in G$ such that $J_0 h = J_0$.
\end{prop}

\begin{proof}In case (1), $J_0=J_\beta$ which is nonpermutable by assumption.
In case (2), first suppose that $J_0 h = J_0$, where $h\in H_0$. Then $\zeta$ is resistant to the shift by $h$ because it contains $\gamma$ as a subblock centered at $e$. To complete the proof, it remains to show that $J_0 h\neq J_0$, for any $h\notin H_0$. Suppose $J_0$ is permuted as a result of right multiplication by some $h\notin H_0$. By the definition of $H_0$, such a permutation either moves an element of $J_\beta g_1$ out of $J_\beta g_1$ or moves an element of $J_\beta g_2$ out of $J_\beta g_2$ (or both). There are thus two possibilities: 
\begin{enumerate}
	\item[(a)] some element of $J_\beta g_2$ is moved to $J_1$, or
	\item[(b)] $J_\beta g_2$ is permuted within itself (i.e., $h\in H_{g_2}=g_2^{-1}
	H_\beta g_2$) and some element of $J_\beta g_1$ is moved to $J_\gamma$. 
\end{enumerate}

In case (a), some element of $J_1$ must be moved to $J_\beta g_2$, which implies that $h\in J_1^{-1}J_\beta g_2$, equivalently, $g_2\in J_\beta^{-1}J_1h$. On the other hand, since $|J_1|>|J_\beta|$, some element of $J_1$ must be moved within $J_1$, implying that $h\in J_1^{-1}J_1$. Eventually, $g_2\in J_\beta^{-1}J_1J_1^{-1}J_1$, which contradicts~\eqref{g2} and eliminates case (a).

Case (b) is eliminated similarly: Some element of $J_\gamma$ must be moved to $J_\beta g_1$ which implies that $g_1\in J_\beta^{-1}J_\gamma h$. Since $|J_\gamma|>|J_\beta|$, some element of $J_\gamma$ must be moved within $J_\gamma$, implying that $h\in J_\gamma^{-1}J_\gamma$. Eventually, $g_1\in J_\beta^{-1}J_\gamma J_\gamma^{-1}J_\gamma$, a contradiction to~\eqref{g1}. The proof is now complete.
\end{proof}

\subsection{Appending a block encoding the index of a shape}\label{enco}
We will now equip the unambiguous marker $\zeta$ with an additional ``appendage'' responsible for encoding a natural number $t\in\{1,2,\dots,r_\eps\}$ representing the index of a shape $S_1,S_2,\dots,S_{r_\eps}$ of $\TT$. 

The set $\bar X$ is infinite (for example, as a consequence of positive entropy); therefore, there exists a finite subset $L\ni e$ of $G$ such that there are at least $r_\eps$ different $\bar X$-admissible blocks over $L$. We choose exactly $r_\eps$ such blocks and denote them by $\bar\kappa_1,\bar\kappa_2,\dots,\bar\kappa_{r_\eps}$. We fix an element $g_0\in G$ so that $Lg_0$ is $\bar M$-apart from the domain $Z$ of the unambiguous marker $\zeta$. Finally, for each $t\in\{1,2,\dots,r_\eps\}$, we define the marker $\zeta_t$ as the block over the set $Z_0=Z\cup Lg_0$ such that
$$
\zeta_t|_{Z}=\zeta, \ \ \zeta_t|_{Lg_0}\approx\bar\kappa_t.
$$
This step completes the construction of the markers (see Figure~\ref{zetat}).

\begin{rem}
Although the domain $L$ of $\bar\kappa_t$ is shown in Figure \ref{zetat} as relatively small, it may be much larger than $\underline B$ and $\underline\Gamma$. The proportion depends on the entropy gap $\htop(X)-h(\nu)$. A small gap yields small $\eps$ (see \eqref{eps}) and thus large $r_\eps$ (see footnote~5) causing the set $L$ to be large.
\end{rem}

\subsection{Reinforcing almost $\bar X$-admissible patterns}
In the construction of the isomorphism $\Psi$ we will need to glue patterns that are mostly $\bar X$-admissible except that they contain the markers. We will call such patterns ``almost $\bar X$-admissible''. The following is the formal definition:

\begin{defn}\label{al}Let $E\subset G$ be any set. An element $x\in X$ is called \emph{almost $\bar X$-admissible with the exceptional set $E$} if there exist a maximal $W$-separated set $V$ such that $x|_{RV\setminus E}=\rho_V|_{RV\setminus E}$. If furthermore $A\subset G$ is such that $R^2E\subset A$ then the pattern $\alpha=x|_A$
is called an \emph{almost $\bar X$-admissible pattern with the exceptional set $E$}
(see Figure~\ref{almost}).\footnote{If $x$ is an almost $\bar X$-admissible element with the exceptional set $E$, there is no guarantee that $x|_{G\setminus E}$ is $\bar X$-admissible. There may be no way to insert missing copies of $\rho$ (with domains intersecting $E$) without changing $x$ over the set $M^2R^2E\setminus E$. The same applies to almost $\bar X$-admissible patterns.}
\end{defn}

\begin{figure}[h]
\includegraphics[width=\columnwidth]{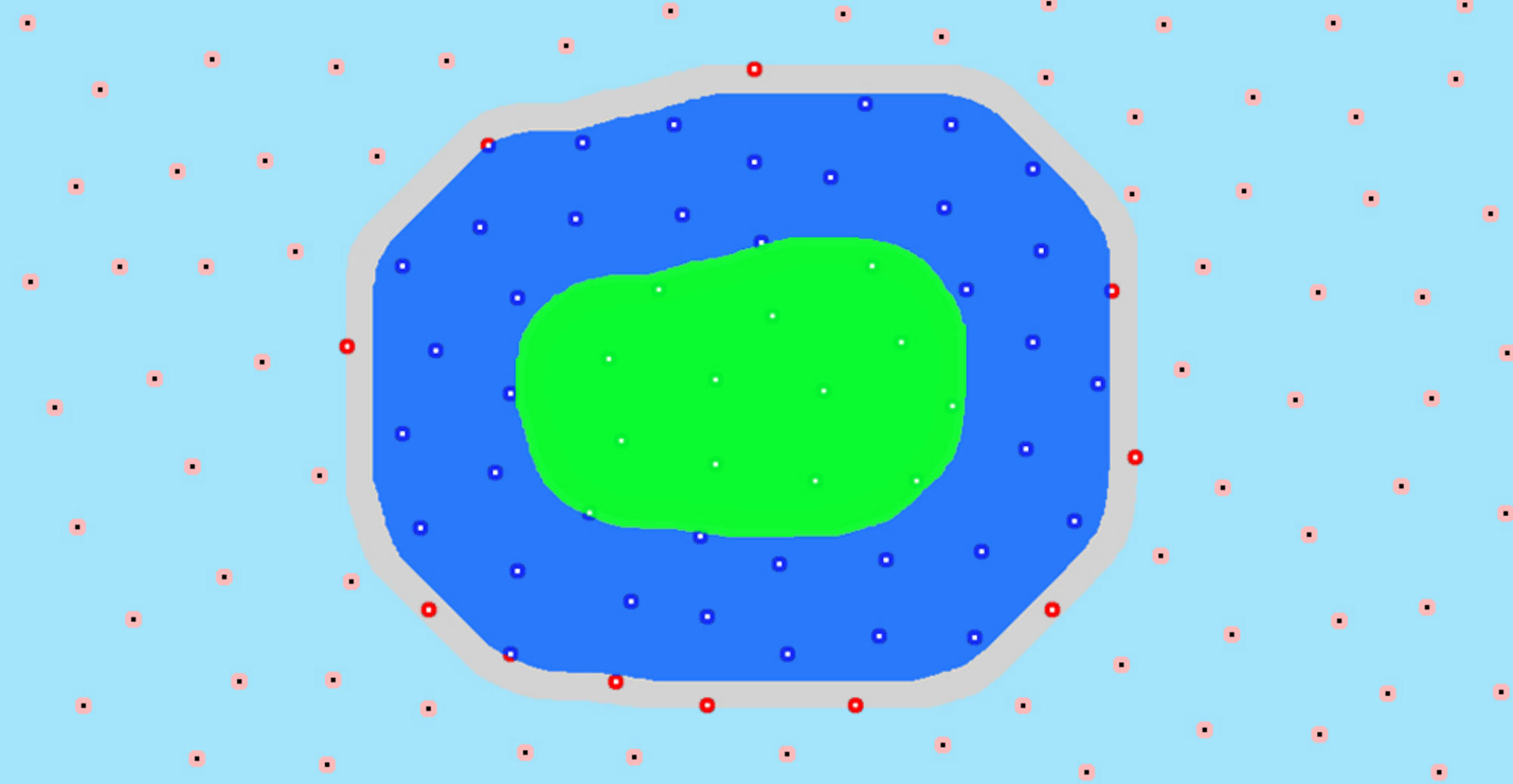}%
\caption{The whole area is an almost $\bar X$-admissible element $x$ with the exceptional set $E$ (the green patch at the center). The grid $\rho_V$ is shown as small circles (rose, red and dark blue). The pattern $\alpha$ (with domain $A$ -- union of blue and green) is cut out from $x$ and inherits part of the grid disrupted within the green set $E$ (complete and incomplete dark blue circles). The reinforced pattern $\lceil\alpha\rceil$ includes $\alpha$ and the copies of $\rho$ from the grid whose domains $Rv$ are not contained in $A$ but are not $M$-apart from $A$ (red and red/blue circles). The domain of the binary pattern $\chi_\alpha$ is $RM^2A$ (green/blue/gray) and the symbols $1$ of $\chi_\alpha$ are marked by the white dots within the green, blue and gray areas.}
\label{almost}
\end{figure}

Clearly, every $\bar X$-admissible pattern is almost $\bar X$-admissible (with $E=\varnothing$). Also, if $\alpha$ with domain $A\ni e$ is an $X$-admissible block then the enhanced block $\underline\alpha$ (with domain $\underline A$) is almost $\bar X$-admissible with the exceptional set $M^2A$ (green and yellow area in Figure~\ref{enha}), because it is cut from $\bar x^*_\alpha$ obtained by inserting $\alpha$ into $\bar x^*\in\bar X$.

Before we glue almost $\bar X$-admissible patterns, we need to prepare them in a special way, which we call ``reinforcing''.

Let $\alpha$ be an almost $\bar X$-admissible pattern with the exceptional set $E$. Let $x\in X$ and $V$ be as in Definition~\ref{al}. The copies of $\rho$ that make up the grid $\rho_V$ can be classified as follows (see Figure~\ref{almost}):
\begin{enumerate}
	\item[a)] with domains $M$-apart from $A$ (pink), 
	\item[b)] with domains disjoint from $A$ but not $M$-apart from $A$ (red),
	\item[c)] with domains partly contained in $A$ (red/blue),
	\item[d)] with domains fully contained in $A$ and disjoint from $E$ (blue),
	\item[e)] with domains intersecting $E$ (incomplete blue and completely missing, whose centers are marked as the white dots inside $E$).
\end{enumerate}
Since $R^2E\subset A$, the classes c) and e) are disjoint, and hence all classes are pairwise disjoint. We need to create a pattern that combines $\alpha$ with all copies of $\rho$ from classes b) and c). The copies of $\rho$ from class d) are included in $\alpha$, so it does not matter whether we add them or not. However, adding them will make the description of the combined pattern a bit easier. Namely, copies of $\rho$ belonging to the classes b), c), and d) can be identified by the property that their centers lie in the set $RM^2A\setminus RE$. We let
\begin{equation}\label{reinA}
\lceil A\rceil=A\cup R(V\cap (RM^2A\setminus RE))\subset R^2M^2A
\end{equation}
and we define the \emph{reinforced pattern} as 
$$
\lceil\alpha\rceil=x|_{\lceil A\rceil}.
$$
Clearly, this pattern is $X$-admissible, as it occurs in $x$. Observe that
$$
\lceil\alpha\rceil|_A=\alpha, \text{ \ and \ } \lceil\alpha\rceil|_{Rv}\approx\rho\text{ for all }v\in V\cap(RM^2A\setminus RE).
$$

We also let
$$
\chi_\alpha=\mathbbm1_V|_{RM^2A}.
$$ 
\begin{rem}\label{takise}
The binary pattern $\chi_\alpha$ carries more information than just the set $V\cap RM^2A$. It also ``memorizes'' (in the form of symbols $0$) where the elements of $V$ {\bf do not occur} within $RM^2A$. 
\end{rem}

The following property of almost $\bar X$-admissible patterns is crucial for the efficiency of our markers.

\begin{lem}\label{nobeta}
If $\alpha$ with domain $A$ is an almost $\bar X$-admissible pattern with the exceptional set $E$, and the basic marker $\beta$ occurs in $\alpha$ centered at some $g$ then $g\in BE$ (that is, the domain of each occurrence of $\beta$ within $\alpha$ must intersect $E$).
\end{lem}
\begin{proof}
Let $x$ and $V$ be as in Definition~\ref{al}. Suppose that $Bg\subset A$ and $Bg\cap E=\varnothing$. By property (I) of the maximal $W$-separated sets, $W^2g$ contains at least one element $v\in V$ and then $Rv\subset RW^2g\subset Bg$. This shows that at least one complete copy of $\rho$ from the grid $RV$ falls within $Bg$. Since $Bg\subset A$ and $Bg\cap E=\varnothing$ this complete copy of $\rho$ occurs in $\alpha$. As $\rho$ does not occur as a subblock of $\beta$ (since $\beta$ is $X'$-admissible, while $\rho$ is not), the latter cannot occur in $\alpha$ over $Bg$.
\end{proof}

\subsection{Gluing almost $\bar X$-admissible patterns} 

This subsection is devoted to proving the following important lemma:

\begin{lem}\label{glu}Let $\alpha_i$ be finitely or countably many almost $\bar X$-admissible patterns with domains $A_i$ and exceptional sets $E_i$, $i=1,2,\dots,k$ (or $i=1,2,\dots$). Suppose that all the sets $A_i$ are pairwise $\bar M$-apart. Then the pattern $\alpha$ over $A=\bigcup_iA_i$ such that, for each $i$, $\alpha|_{A_i}=\alpha_i$, is almost $\bar X$-admissible with the exceptional set $E=\bigcup_iE_i$.
\end{lem}

\begin{rem}\label{spxb}
In particular, the lemma applied to $\bar X$-admissible patterns implies that the system $(\bar X,G)$ has the specification property with the margin $\bar M$.
\end{rem}

\begin{proof}[Proof of Lemma~\ref{glu}]
We need to construct an almost $\bar X$-admissible element $x\in X$ with the exceptional set $E$ such that, for each $i$, $x|_{A_i}=\alpha_i$. This includes creating a maximal $W$-separated set $V$ satisfying $x|_{RV\setminus E}=\rho_V|_{RV\setminus E}$. 

Since the domains $A_i$ are $\bar M$-apart, where $\bar M=W^3RM^2$ (see \eqref{barm}), and $W\supset MR$ (see~\eqref{H3}), the domains $\lceil A_i\rceil$ of the reinforced patterns $\lceil\alpha_i\rceil$ are $M$-apart (see~\eqref{reinA}) (in fact, they are apart by a much larger set $W^2M$). By the specification property of $X$, the patterns $\lceil\alpha_i\rceil$ jointly form an $X$-admissible pattern. Furthermore, the domains of the binary patterns $\chi_{\alpha_i}$ are the sets $RM^2A_i$ and they are $W^3$-apart from each other. 
By the specification property of $\Omega_W$ (see Lemma~\ref{specK}), there exists a maximal $W$-separated set $V$ whose indicator function $\omega=\mathbbm1_V$ satisfies, for each $i$,
$$
\omega|_{RM^2A_i}=\chi_{\alpha_i}.
$$
If $v\in V$ falls within $RM^2A_i$ for some $i$, then $\chi_{\alpha_i}(v)=1$. Let $V^*=V\setminus\bigcup_iRM^2A_i$. If $v\in V^*$ then $Rv$ is $M$-apart from each of the sets $A_i$. The set $Rv$ is also $M$-apart from any other set of the form $Rv$ with $v\in V$, because $W\supset MR$. Thus, $Rv$ with $v\in V^*$ is $M$-apart from $\lceil A_i\rceil$ for all $i$. This implies that the patterns $Rv$ with $v\in V^*$ jointly with the patterns $\lceil\alpha_i\rceil$ form an $X$-admissible pattern. Thus, there exists an $x\in X$ such that $x|_{\lceil A_i\rceil}=\lceil\alpha_i\rceil$ (in particular, $x|_{A_i}=\alpha_i$) for each $i$, and $x|_{Rv}\approx\rho$ for all $v\in V^*$. Taking into account the copies of $\rho$ occurring within the reinforced blocks $\lceil\alpha_i\rceil$, we can see that $x$ satisfies the desired condition $x|_{RV\setminus E}=\rho_V|_{RV\setminus E}$, i.e., it is almost $\bar X$-admissible with the exceptional set $E$ (see Figure~\ref{gluing}).
\end{proof}

\begin{figure}[h]
\includegraphics[width=\columnwidth]{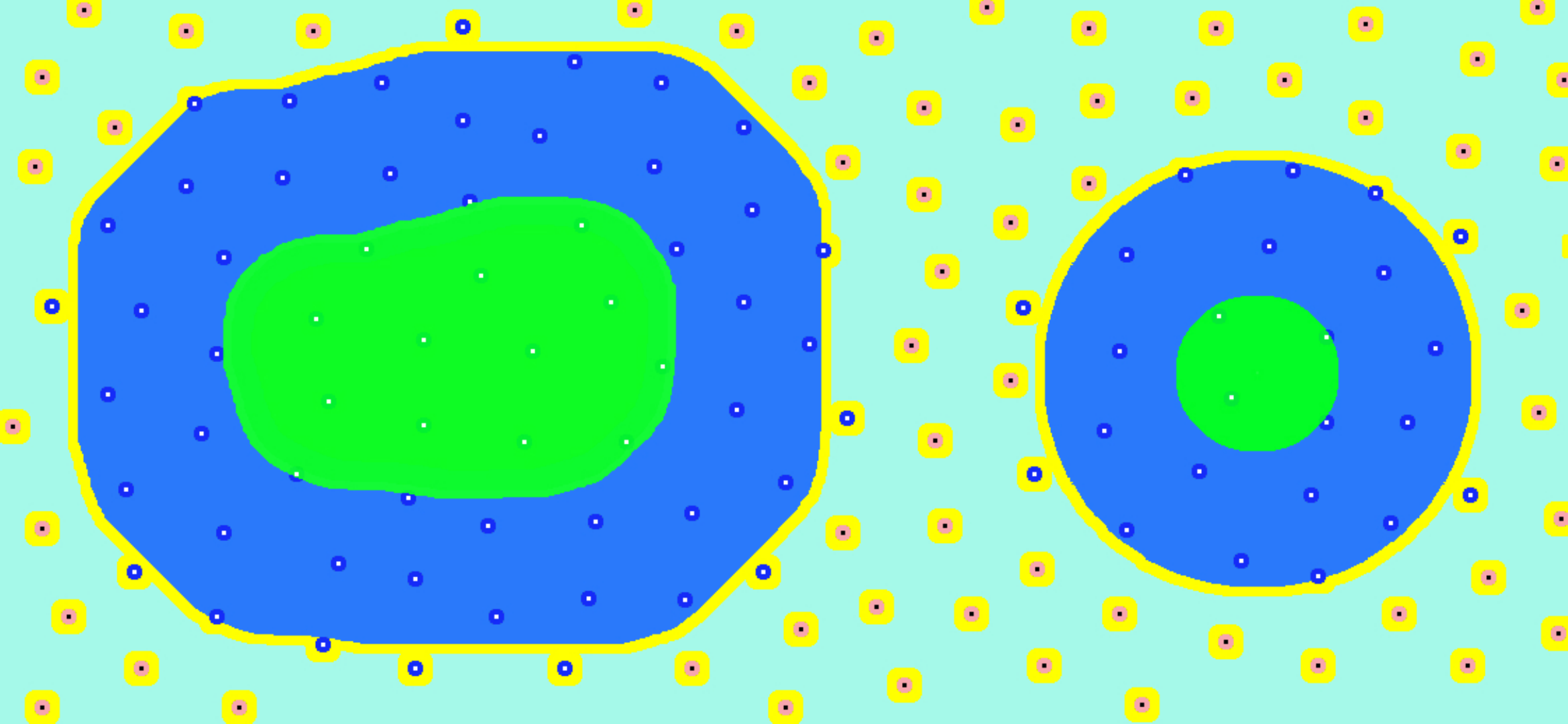}
\caption{The binary patterns $\chi_{\alpha_i}$ (the white dots inside the blue and green areas) are completed to the indicator function of a maximal $W$-separated set $V$ (white and black dots). The reinforced patterns $\lceil\alpha_i\rceil$ (shown in blue) and the (pink) copies of $\rho$ centered at the added (black) elements of $V$ are $M$-apart from each other (the yellow margins around them are disjoint). Thus their union is $X$-admissible and can be realized in an $x\in X$.} 
\label{gluing}
\end{figure}

\section{Construction of the isomorphism}
We pass to the construction of the desired isomorphism $\Psi:(Y,\nu,G)\to (X,\mu,G)$. For $\nu$-almost every $y\in Y$ we need to create a symbolic element $x=\Psi(y)\in X$, so that $\Psi$ is invertible, measurable, and shift-equivariant. 

\subsection{Review of what has been done so far}
Let us briefly review the situation in which we begin the construction.
The review corresponds to the sketch in subsection~\ref{skecz}, but now we can establish some important details that have been omitted in the sketch.
We are given a free \tl\ system $(Y,G)$ and an ergodic measure $\nu$ on $Y$ with $h(\nu)<\htop(\bar X)$, where $(\bar X,G)$ is a special proper subsystem of $(X,G)$ (both $\bar X$ and $X$ are subshifts over the alphabet $\Lambda$).
We also have a uniquely ergodic \tl\ subshift $(Y',G)$ with an alphabet $\Delta$, whose unique \im\ is isomorphic to $\nu$, thus $\htop(Y')<\htop(\bar X)$.
Let us denote
\begin{equation}\label{del}
\delta=\tfrac12(\htop(\bar X)-\htop(Y')).
\end{equation}
In addition, we have defined the ``basic marker'' $\beta$, which does not occur in $\bar X$ and we have established the margin $\bar M$ for gluing almost $\bar X$-admissible patterns, described in Lemma~\ref{glu}.

Let us now specify $\eps$. We let it be so small that 
\begin{equation}\label{eps}
5|\bar M|\eps<1\ \ \text{and}\ \ \Theta^\eps<2^\delta,
\end{equation} 
where $\Theta=|\Delta|^2|\Lambda|^{5|\bar M|}$. The choice of $\eps$ determines a natural number $r_\eps$ that equals the number of shapes of a quasitiling associated with the subshift $(Y,G)$ (to be discussed in a moment).
Once we know the value of $r_\eps$, we can define the domain $L$ and $r_\eps$ different $\bar X$-admissible blocks $\bar\kappa_t$ over $L$. Finally, we define the ``ultimate'' markers $\zeta_t$ with domain $Z_0 = Z\cup Lg_0$  so that $\zeta_t|_Z=\zeta$ and $\zeta_t|_{Lg_0}\approx\overline\kappa_t$.

Next, we have the (commuting) diagram of factor maps:
\begin{align*}
Y'\overset{\pi}{\longleftarrow}Y\overset{\phi}{\longrightarrow}\TT\overset{\tilde\phi}{\longrightarrow}&\tilde\TT\overset{\hat\phi}{\longrightarrow}\hat\TT
\\[-15pt]
\rotatebox{330}{$\xrightarrow{\makebox[1.5cm]{$\tilde\Phi$}}$}\!\!&\
\rotatebox{90}{$\!\!\!\!\!\!\!\!\!\longrightarrow$}
\\[-10pt]
&\,\tilde\TTT
\end{align*}
where 
\begin{itemize}
	\item $\pi$ is a measure-theoretic isomorphism between $\nu$ on $Y$ and the unique \im\ on $Y'$,
	\item $\phi$ is a \tl\ factor map onto a dynamical quasitiling $\TT$ with $r_		\eps$ shapes selected from the F\o lner \sq\ $\F=(F_n)_{n\ge1}$; the smallest shape is $F_{n_0}$ and we can choose $n_0$ as large as we wish,  
	\item $\tilde\phi$ is a \tl\ factor map onto a disjoint dynamical quasitiling $\tilde\TT$ whose tiles are $(1-2\eps)$-subsets of those of $\TT$ and all shapes $\tilde S$ of $\tilde\TT$ contain the set $K=\bar M^3Z_0$ (see Lemma~\ref{Fintildas} (iv)) -- note that then $\bar M^2Z_0\subset \tilde S_{\bar M}$,
	\item $\hat\phi$ is a measure-theoretic factor map to an improper dynamical tiling $\hat\TT$ whose tiles are $2\eps$-enlargements of those of $\tilde\TT$, 
	\item $\tilde\Phi$ is a \tl\ factor onto a F\o lner, congruent system of disjoint quasitilings $\tilde\TTT$,
	\item the upward arrow is the projection of $\tilde\TTT$ onto the first ``level''
$\tilde\TT_1$ which coincides with $\tilde\TT$ (see Corollary~\ref{cor_quasitiling}).
\end{itemize}
By choosing $n_0$ large enough and by Lemma~\ref{sepcen}, we can arrange that 
\begin{equation}\label{dopo}
\text{for any $\tilde\T\in\tilde\TT$ the set $C(\tilde\T)$ of centers of the tiles of $\tilde\T$ is $Z_0^{-1}Z_0$-separated}.
\end{equation}
Since we aim to build a measure-theoretic map $\Psi:Y\to X$, defined $\nu$-almost everywhere, we will build $x=\Psi(y)$ for $y$ belonging to a subset $Y_0\subset Y$ such that $\nu(Y_0)=1$, $\pi|_{Y_0}$ is defined and injective, and the composition $\hat\phi\circ\tilde\phi\circ\phi(y)$ is defined. Clearly, such a set $Y_0$ exists.

\medskip\noindent
\subsection{Injection from $Y'$-admissible blocks to $\bar X$-admissible blocks}
Since all shapes of $\TT$ are members of the F\o lner \sq\ $\F=(F_n)_{n\ge1}$, with indices at least $n_0$, by \eqref{del} and Definition~\ref{htop}, choosing $n_0$ large enough we can assure that for all shapes $S$ of $\TT$, we have
\begin{equation}\label{en1}
|\B_{\bar X}(S)|>|\B_{Y'}(S)|\cdot2^{\delta|S|},
\end{equation}
where $\B_{\bar X}(S)$ and $\B_{Y'}(S)$ denote the collections of blocks with domain $S$ that occur in $\bar X$ and $Y'$, respectively. 
Since any shape $\hat S$ of $\hat\TT$ is a $2\eps$-enlargement of some $\tilde S\subset S$,
we have
\begin{equation}\label{en2}
|\B_{Y'}(\hat S)|\le|\B_{Y'}(\tilde S)|\cdot |\Delta|^{2\eps|\tilde S|}\le|\B_{Y'}(S)|\cdot |\Delta|^{2\eps|S|}.
\end{equation}
Since $\tilde S$ is a $(1-2\eps)$-subset of $S$, by Proposition~\ref{ksets} (3), the $\bar M$-core of $\tilde S$, $\tilde S_{\bar M}$, is a $(1-4|\bar M|\eps)$-subset of $S$. Recall that $\bar M^2Z_0\subset \tilde S_{\bar M}$. Define 
$$
\tilde S_0=\tilde S_{\bar M}\setminus\bar M^2Z_0.
$$ 
The ``hole'' $\bar M^2Z_0$ at the center of $\tilde S_0$ will later be used for the insertion of a marker.
Because we subtract a constant set (not depending on $n_0$), by choosing $n_0$ large enough, we can arrange that
$\tilde S_0$ is a $(1-5|{\bar M}|\eps)$-subset of $S$. Then
\begin{equation}\label{en3}
|\B_{\bar X}(\tilde S_0)|\cdot |\Lambda|^{5|{\bar M}|\eps|S|}\ge|\B_{\bar X}(S)|.
\end{equation}
Combining the inequalities \eqref{en1}, \eqref{en2} and \eqref{en3}, we get
\begin{equation}\label{en4}
|\B_{Y'}(\hat S)|<|\B_{\bar X}(\tilde S_0)|\cdot\frac{|\Delta|^{2\eps|S|}\cdot|\Lambda|^{5|{\bar M}|\eps|S|}}{2^{\delta|S|}}=|\B_{\bar X}(\tilde S_0)|\cdot\Bigl(\frac{\Theta^\eps}{2^\delta}\Bigr)^{|S|}.
\end{equation}
Recall that $\eps$ was selected so small that $\Theta^\eps<2^\delta$ (see~\eqref{eps}), and thus
\begin{equation}\label{en5}
|\B_{Y'}(\hat S)|<|\B_{\bar X}(\tilde S_0)|.
\end{equation}
Therefore, whenever $\tilde S$ is a shape of the quasitiling $\tilde\TT$ and $\hat S$ is a shape of the tiling $\hat\TT$, such that there exist two tiles sharing a common center: $\tilde T=\tilde Sc$ of $\tilde\T\in\tilde\TT$ and $\hat T=\hat Sc$ of $\hat\T=\hat\phi(\tilde\T)$, then there exists an injection $\theta_{\hat S,\tilde S}$ from $\B_{Y'}(\hat S)$ into $\B_{\bar X}(\tilde S_0)$. 

Now, given $y\in Y_0$, we let $\tilde\T=\tilde\phi\circ\phi(y)$ and
$\hat\T=\hat\phi(\tilde\T)$, and for each pair of tiles sharing a common center: $\tilde T=\tilde Sc$ of $\tilde\T$ and $\hat T=\hat Sc$ of $\hat\T$, we define the restriction 
$$
x|_{\tilde T_0}\approx \theta_{\hat S,\tilde S}(\alpha),
$$ 
where $\tilde T_0 = \tilde S_0c$ (equivalently, $\tilde T_0=\tilde T_{\bar M}\setminus\bar M^2Z_0c$) and $\alpha\in\B_{Y'}(\hat S)$ is such that $\alpha\approx y'|_{\hat T}$ where $y'=\pi(y)$. 

\subsection{Marking the centers and encoding the shapes} 
In this manner we have completely encoded $y'$ (and hence also $y=\pi^{-1}(y')$) within $x$. For this coding to be invertible (that is, allowing to recover $y'$ from $x$), we need to mark in $x$ the centers and memorize the shapes of the tiles of $\T=\phi(y)$. Once this information is included in $x$, based on $x$ we will be able to recover the centers and shapes of both $\tilde\T$ and $\hat\T$, because these quasitilings are determined by $\T$. This is where the markers $\zeta_t$ come in handy. If $c$ is the center of a tile $T=S_tc$ of $\T$ with shape $S_t$ (recall that the shapes of $\TT$ are $S_1,S_2,\dots,S_{r_\eps}$), $c$ is also the center of the corresponding tiles $\tilde T=\tilde Sc$ of $\tilde\T$ and $\hat T=\hat Sc$ of $\hat\T$. In this case, we additionally define
$$
x|_{Z_0c}\approx\zeta_t.
$$
Notice that, by the definition of $\tilde T_0$, the domains $\tilde T_0$ of $\theta_{\hat S,\tilde S}(\alpha)$ and $Z_0c$ of the marker $\zeta_t$ inserted into the ``hole'' inside $\tilde T_0$ are $\bar M$-apart.
Recall also that while building the marker $\zeta_t$ we took care that all components of $\zeta_t$ are almost $\bar X$-admissible and lie $\bar M$-apart from each other. The construction of the part of $x$ defined so far (in the more complicated permutable case) is shown in Figure \ref{codingarea}.

\begin{figure}[ht]
\includegraphics[width=\columnwidth]{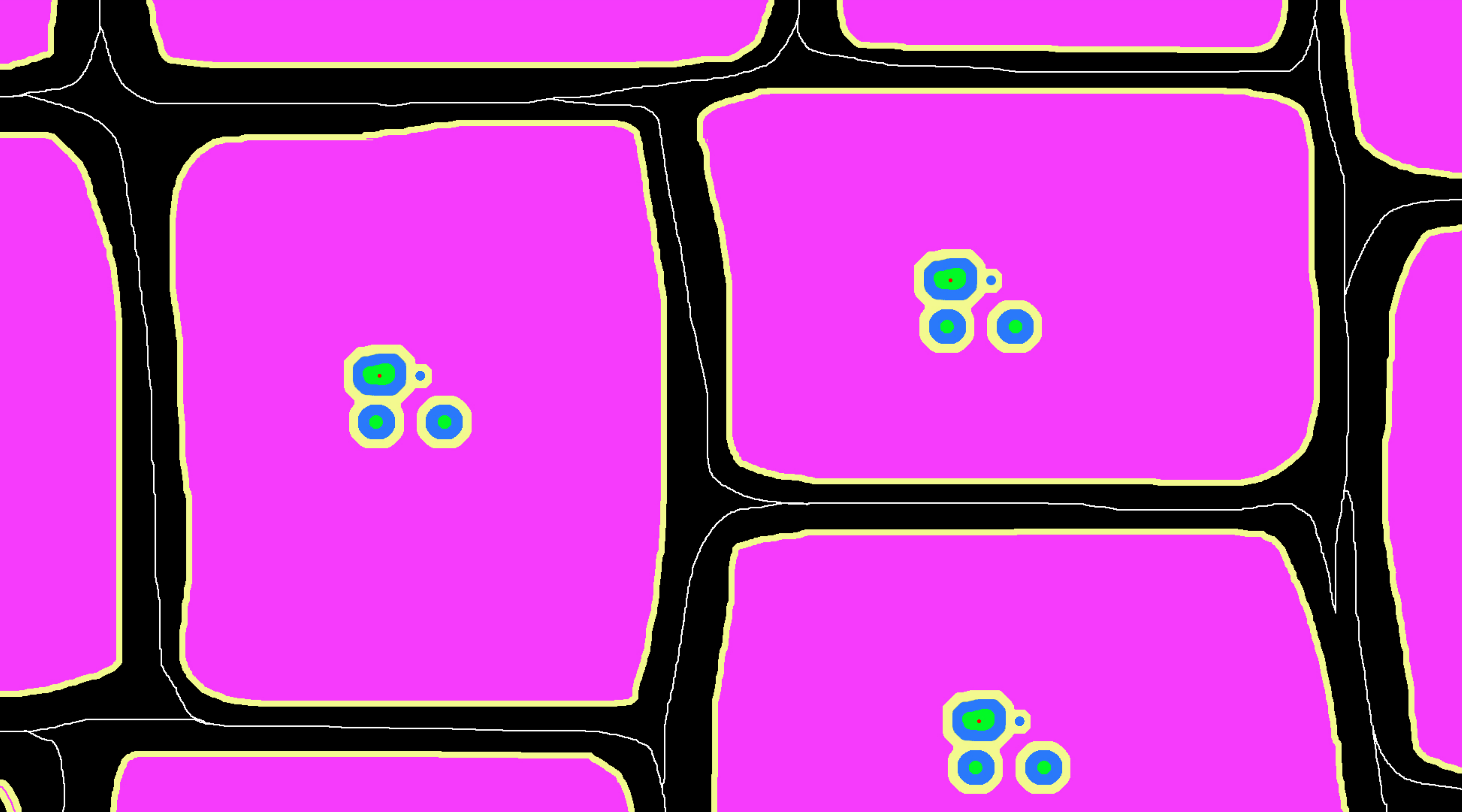}
\caption{\footnotesize The blocks $\alpha$ occurring in $y'$ over the tiles $\hat T=\hat Sc$ (framed by the thin white lines) are encoded in $x=\Psi (y)$ as $\bar X$-admissible blocks $\theta_{\hat S,\tilde S}(\alpha)$ over the sets $\tilde T_0$ (shown as violet areas). At the centers of these tiles there occur the markers $\zeta_t$, each with four components (rescaled from Figure~\ref{zetat}). All violet ``coding areas'' and all components of the markers $\zeta_t$ are almost $\bar X$-admissible and $\bar M$-apart from each other (separated by the yellow margins). The yellow margins and the black gaps between the tiles of $\tilde\T$ remain temporarily unfilled in $x$.}\label{codingarea}
\end{figure}
\smallskip

\subsection{Measurable and shift-equivariant filling of the remaining part}\label{killer}
By now, for each $y\in Y_0$, we have defined the restriction of the future image $x=\Psi(y)$ to the set
$$
\bigcup_{\tilde T\in\tilde\T}(\tilde T_0\cup Z_0c_T),
$$ 
where $\tilde\T=\tilde\phi\circ\phi(y)$, $c_T$ denotes the center of $\tilde T$
and $\tilde T_0=\tilde T_{\bar M}\setminus \bar M^2Z_0c_T$.
Let us denote the partially filled $\Psi(y)$ as $\Psi_0(y)$ (one can imagine an additional ``empty'' symbol $\square$ at all unfilled positions, so that $\Psi_0(y)$ is symbolic, with a finite alphabet). Notice that $\Psi_0(y)$ is determined on each tile of $\tilde\T$ via a finitary code\footnote{In symbolic dynamics, a finitary code is a mapping between symbolic spaces $Y$ and $X$ in which each coordinate of the image $x$ can be determined by viewing a finite window in $y$, where the size of the window depends on the coordinate. Finitary codes defined almost everywhere on $Y$ are measure-theoretic factor maps.} relying on examining $y'$, $\T$, $\tilde\T$ and $\hat\T$ (all of which are images of $y$ via measure-theoretic factor maps) in a finite window. This implies that $\Psi_0:Y_0\to(\Lambda\cup\{\square\})^G$ is measurable and shift-equivariant.
\smallskip

\noindent{\emph{Step 1}.} We will work with the permutable case, as the non-permutable case is similar (and in fact simpler). In this step, we will just fill the margins around the markers, i.e., we will define $\Psi_1(y)$ so that in addition to the symbols from $\Lambda$ already assigned by $\Psi_0$, it also assigns symbols from $\Lambda$ over the sets $\bar M^2Z_0c_T\setminus Z_0c_T$.

We temporarily fix an element $y^*\in Y_0$ and a tile $\tilde T^*$ of the quasitiling $\tilde\T^*=\tilde\phi\circ\phi(y^*)$. By shifting $y^*$ we can arrange that $\tilde T^*$ is centered at $e$. Within $\tilde T^*$ the element $\Psi_0(y^*)$ has been partly filled by five patterns: the $\bar X$-admissible ``coding pattern'' of the form $\theta_{\hat S,\tilde S}(\alpha)$, $\underline\gamma$ centered at $e$, which is almost $\bar X$-admissible with the exceptional set $M^2\Gamma$, two copies of $\underline\beta$, centered at $g_1$ and $g_2$, both being almost $\bar X$-admissible with the exceptional sets $M^2Bg_1$ and $M^2Bg_2$, respectively, and $\bar\kappa_t$ centered at $g_0$, which is $\bar X$-admissible. The domains of all these patterns are $\bar M$-apart from each other. So, by Lemma~\ref{glu}, these five patterns jointly constitute an almost $\bar X$-admissible pattern with the exceptional set 
\begin{equation}\label{Ewzor}
E=M^2(\Gamma\cup B\{g_1,g_2\})
\end{equation}
(note that the set $E$ does not depend on $y^*$ or $\tilde T^*$ and $Z=B^2E$, see~\eqref{domzeta}).
This implies that there exists an almost $\bar X$-admissible element $x_1^*\in X$ with the exceptional set $E$ that matches the five patterns over their domains. Let us denote
$$
\xi_{1,y^*}=x_1^*|_{\tilde T^*_{\bar M}}.
$$
By definition, $\xi_{1,y^*}$ is an almost $\bar X$-admissible pattern with the exceptional set~$E$. 

Now comes a crucial detail in the construction, responsible for measurability and shift-equivariance. Namely, whenever for some $y\in Y_0$, some tile $\tilde T$ of $\tilde\T=\tilde\phi\circ\phi(y)$ has the same shape as $\tilde T^*$ and satisfies $\Psi_0(y)|_{\tilde T}\approx\Psi_0(y^*)|_{\tilde T^*}$, we put
$$
\Psi_1(y)|_{\tilde T_{\bar M}}\approx\xi_{1,y^*}
$$ 
Since there are finitely many shapes of $\tilde\TT$ and finitely many possible patterns over each shape, it suffices to repeat the above procedure some finite number of times, with different choices of $y^*$ and $\tilde T^*$, until $\Psi_1(y)$ is defined for all $y\in Y_0$ on the $\bar M$-cores of all tiles $\tilde T$ of $\tilde\T=\tilde\phi\circ\phi(y)$. Outside these cores $\Psi_1(y)$ is filled with the ``empty'' symbols $\square$. This concludes the first step of the filling procedure.
It is clear that the procedure just described depends on $\Psi_0(y)$ via a bounded horizon code (and thus is a \tl\ factor map). Therefore, $\Psi_1:Y_0\to(\Lambda\cup\{\square\})^G$ is measurable and shift-equivariant. Notice that $\Psi_1$ preserves the symbols from $\Lambda$ assigned by $\Psi_0$ and just replaces the ``empty'' symbols $\square$ appearing over the margins around the markers with symbols from $\Lambda$.

\smallskip\noindent\emph{Inductive step.} We will now fill the gaps between the tiles of $\tilde\T_1$. Recall that each $y\in Y_0$ determines, via a \tl\ factor map $\tilde\Phi$, a congruent \sq\ of disjoint quasitilings $(\tilde\T_n)_{n\ge1}\in\tilde\TTT$ such that $\tilde\T_1=\tilde\T=\tilde\phi\circ\phi(y)$. Assume that for some $n\ge1$ and every $m=1,2,\dots,n$, we have defined a measurable and shift equivariant map $\Psi_m:Y_0\to(\Lambda\cup\{\square\})^G$, so that, for any $y\in Y_0$ and each tile $\tilde T$ of $\tilde\T_m$, the restriction $\Psi_m(y)|_{\tilde T_{\bar M}}$ is an almost $\bar X$-admissible pattern with the exceptional set 
$$
E_{\tilde T}= \bigcup_{c\in C(\tilde\T)\cap\tilde T} Ec, 
$$ 
($c$ ranges over the set of centers of the tiles of $\tilde\T=\tilde\phi\circ\phi(y)$ contained in $\tilde T$).

We also assume that $\Psi_m(y)$ matches $\Psi_0(y)$ at all coordinates where the latter is filled with symbols from $\Lambda$. Note that for $n=1$ these requirements are fulfilled by $\Psi_1$, so induction is correctly initiated.

To define $\Psi_{n+1}$, we temporarily fix an element $y^*\in Y_0$ such that $e$ is the center of a tile $\tilde T^*$ of the $(n+1)$th ``level'' $\tilde\T^*_{n+1}$ of $\tilde\Phi(y^*)$.

By assumption, $\Psi_n(y^*)$ is filled over the $\bar M$-core of each primary subtile (see Definition~\ref{primary}) $\Breve T^*$ of $\tilde T^*$ by an almost $\bar X$-admissible pattern with the exceptional set $E_{\Breve T^*}$.

Since the $\bar M$-cores of the primary subtiles of $\tilde T^*$ are $\bar M$-apart, the patterns that fill them constitute jointly an almost $\bar X$-admissible pattern with the exceptional set~$E_{\tilde T^*}$. This implies that there exists an almost $\bar X$-admissible element $x_{n+1}^*\in X$ with the exceptional set $E_{\tilde T^*}$ that matches $\Phi_n(y)$ at all positions filled with the symbols from $\Lambda$, contained in $\tilde T^*$.
Let us denote
$$
\xi_{n+1,y^*}=x_{n+1}^*|_{\tilde T^*_{\bar M}}.
$$
By definition, $\xi_{n+1,y^*}$ is an almost $\bar X$-admissible pattern with the exceptional set~$E_{\tilde T^*}$. 

Now, for each $y\in Y_0$, and every tile $\tilde T$ of the $(n+1)$st ``level'' $\tilde\T_{n+1}$ of $\tilde\Phi(y)$ which has the same shape as $\tilde T^*$, and such that $\Psi_n(y)|_{\tilde T}\approx\Psi_n(y^*)|_{\tilde T^*}$, we put
$$
\Psi_{n+1}(y)|_{\tilde T_{\bar M}}\approx\xi_{n+1,y^*}. 
$$ 
After a finite number of repetitions of this process with various elements $y^*$, $\Psi_{n+1}(y)$ is defined for all $y\in Y_0$ on the $\bar M$-cores of all tiles $\tilde T$ of the quasitiling $\tilde\T_{n+1}$ associated with $y$. All remaining positions in $\Psi_{n+1}(y)$ are filled with the ``empty'' symbols $\square$. The map $\Psi_{n+1}:Y_0\to(\Lambda\cup\{\square\})^G$ is measurable and shift-equivariant, and satisfies the inductive assumption for $n+1$. This concludes the induction (see Figure~\ref{final}).

\begin{figure}[h]
\includegraphics[width=\columnwidth]{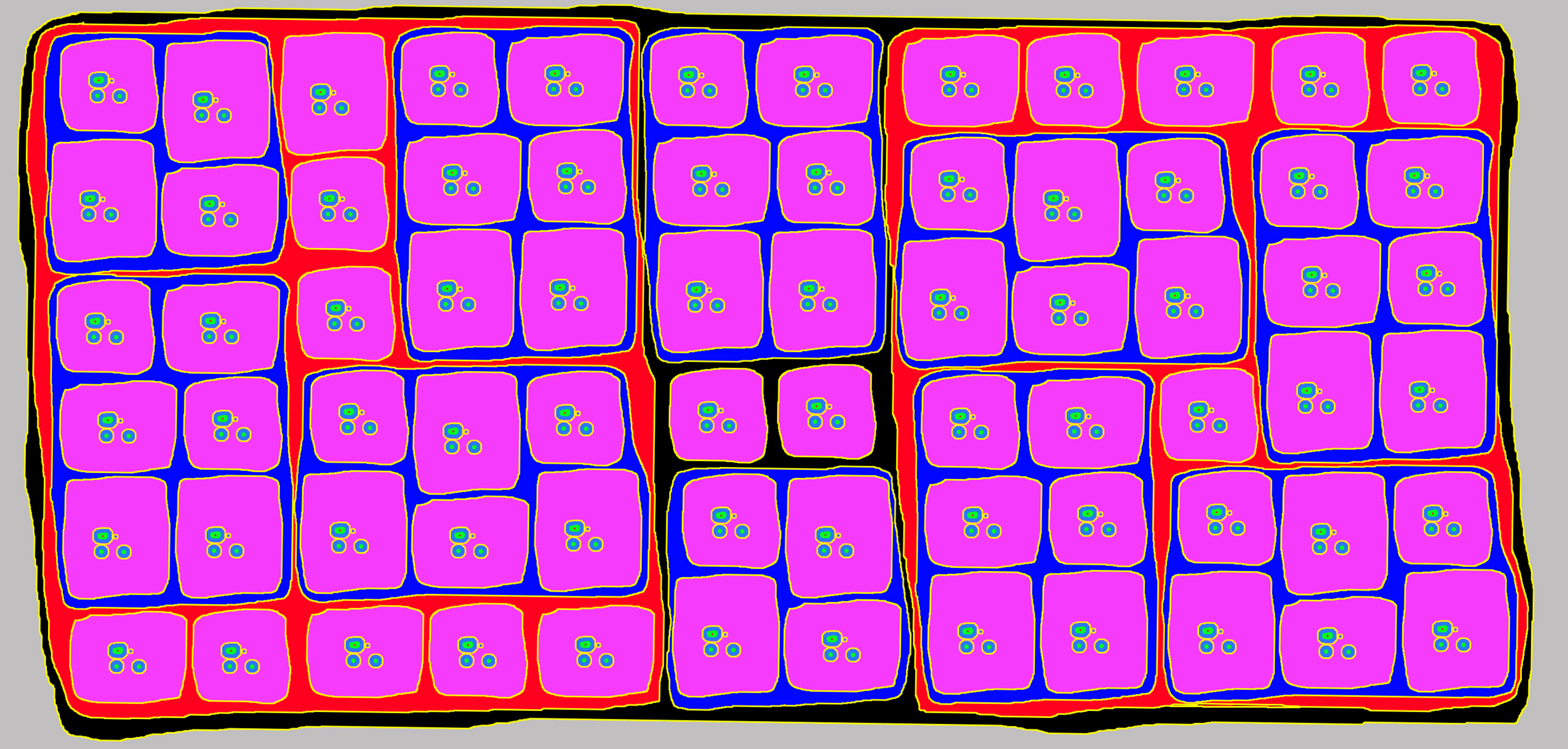}%
\caption{The figure shows a tile $\tilde T$ of $\tilde\T_4$ (black with the thin yellow margin and everything inside). After step 4, $\Psi_4(y)|_{\tilde T}$ is completely defined except over the external yellow margin. Subtiles of levels 3, 2 and 1 are shown in red, blue and violet, respectively. Those on black background (two red, two blue and two violet) are primary subtiles of $\tilde T$. The tiles of level 1 have in their centers the markers $\zeta_t$ and carry in the violet section the encoded information about the symbolic contents of $y'$. The tiny green areas inside the markers are the exceptional sets. The quasitilings of levels higher than 1 are used exclusively to make the filling procedure measurable and shift-equivariant and are not involved in the coding/decoding algorithms.}
\label{final}
\end{figure}

\subsection{The limit map $\Psi$}
Let us now restrict the domain of the maps $\Psi_n$ to $Y_0^*=Y_0\cap Y^*$ (for the meaning and properties of $Y^*$, see Corollary~\ref{cor_quasitiling}). Since $Y^*$ has the full measure for all \im s on $Y$, we have $\nu(Y_0^*)=1$. What we gain from this restriction is that for each $y\in Y_0^*$ the \sq\ of quasitilings $(\tilde\T_n)_{n\ge1}$ associated with $y$ via the map $\tilde\Phi$ is in general position, that is, the central tiles $\tilde T^{(n)}$ of $\tilde\T_n$, grow to the entire group $G$ (see Definition~\ref{gp}). Clearly, the same applies to the cores $\tilde T^{(n)}_{\bar M}$. Because, for each $n$, $\Psi_n(y)|_{\tilde T^{(n)}_{\bar M}}$ is an almost $\bar X$-admissible pattern with the exceptional set $E_{\tilde T^{(n)}}$, the images $\Psi_n(y)$ converge to an almost $\bar X$-admissible element $\Psi(y)\in X$ with the exceptional set 
\begin{equation}\label{Ey}
E_y = \bigcup_{c\in C(\tilde\T)} Ec=EC(\tilde\T).
\end{equation}
where $\tilde\T=\tilde\phi\circ\phi(y)$. Clearly, the limit map $\Psi$ defined on $Y_0^*$ is measurable and shift-equivariant. 

\subsection{Invertibility of $\Psi$}
Because, for each $y\in Y_0^*$, the image $\Psi(y)$ is an almost
$\bar X$-admissible element with the exceptional set $E_y$, Lemma~\ref{nobeta} guarantees that the basic marker $\beta$ occurs in $\Psi(y)$ centered exclusively within $BE_y$ and hence fit within the set $B^2E_y$. 
Recall, however, that 
$$
B^2E_y=B^2EC(\tilde\T)=ZC(\tilde\T)
$$ 
is occupied by copies of the unambiguous markers $\zeta$ (see \eqref{Ey}, \eqref{Ewzor}, \eqref{W6} and \eqref{domzeta}). Thus, all occurrences of $\beta$ in $\Psi(y)$ fit within the copies of the marker $\zeta$, and hence their centers constitute the set $\bigcup_{c\in C(\tilde\T)}J_0 c$ (see \eqref{jzeta2}). 
Suppose that a marker $\zeta$ occurs in $\Psi(y)$ centered at an element $g\in G$.
Then $\beta$ occurs centered at the elements of $J_0 g$. If $J_0 g=J_0 c$ for some $c\in C(\tilde\T)$ then $g=c$, by the unambiguity of the marker $\zeta$ (see Proposition~\ref{unamb}). Otherwise, $J_0 g$ would have to intersect at least two sets of the form $J_0 c_1, J_0 c_2$, $c_1,c_2\in C(\tilde\T)$.
This would imply that $J_0^{-1}J_0 c_1\cap J_0^{-1}J_0 c_2\neq\varnothing$, which is impossible by the $Z^{-1}Z$-separation of $C(\tilde\T)$ (see~\eqref{Jalfa},~\eqref{domzeta} and~\eqref{dopo}). We have shown that markers $\zeta$ occur in $\Psi(y)$ exclusively centered at the centers of the tiles of~$\tilde\T$.

To see that $\Psi$ is invertible, we need to show how $y\in Y^*_0$ can be decoded from its image $\Psi(y)$. By the discussion above, we can locate the centers of the tiles of $\tilde\T=\tilde\phi\circ\phi(y)$ (which are the same as the centers of the tiles of $\T=\phi(y)$) as the centers of the occurrences of $\zeta$ in $\Psi(y)$. 

Each marker $\zeta_t$ is made up of $\zeta$ and the ``appendage'' $\bar\kappa_t$ (where $t\in\{1,2,\dots,r_\eps\}$) centered at $g_0$. Thus, for each center $c\in C(\T)$ we have $\Psi(y)|_{Lg_0c}\approx\bar\kappa_t$, where $t$ coincides with the index of the shape $S_t$ of the tile of $\T$ centered at $c$. This enables us to recover the shape of each tile of $\T$. In other words, $\Phi(y)$ determines $\phi(y)$. Furthermore, $\phi(y)$ determines $\tilde\T=\tilde\phi\circ\phi(y)$ and $\hat\T=\hat\phi\circ\tilde\phi\circ\phi(y)$. That is, for each $c\in C(\T)$ we know the shapes $\tilde S$ and $\hat S$ of the respective tiles $\tilde T$ of $\tilde\T$ and $\hat T$ of $\hat\T$, centered at $c$. Next, for every tile $\tilde T$ of $\tilde\T$ we know that $\Phi(y)|_{\tilde T_0}\approx\theta_{\hat S,\tilde S}(\alpha)$, where $\alpha\approx y'|_{\hat T}$. Because the code $\theta_{\hat S,\tilde S}$ is injective, the block $\alpha$ is now uniquely determined. Since $\hat\T$ is a tiling (i.e., its tiles cover the whole group), $y'$ can be completely reconstructed, tile by tile (the fact that $\hat\T$ is improper is irrelevant). Finally, since $\pi$ is invertible on $Y_0$, we recover $y$ as $\pi^{-1}(y')$. This completes the proof of invertibility of $\Psi$ on $Y_0^*$. So, $\Psi$ is a (measure-theroetic) isomorhism between the systems $(Y,\nu,G)$ and $(X,\mu,G)$, where $\mu$ is the image of $\nu$ via $\Psi$.

{\bf This concludes the proof of the main theorem}.

\newpage
\end{document}